\documentclass[reqno, 10pt]{article}%{amsart}

\usepackage{fullpage}
\usepackage{newlfont}
\usepackage{color}
\usepackage{authblk}

\usepackage{microtype}
\usepackage{subcaption}
\usepackage{graphicx}
\usepackage{booktabs} % for professional tables
\usepackage{algorithm}
\usepackage{algorithmic}
\usepackage{xr-hyper}

\usepackage{amsmath}
\usepackage{amsfonts}
\usepackage{amssymb}
\usepackage{amsthm}
\usepackage{bbm}

\usepackage{hyperref}

\theoremstyle{plain}

\newtheorem{theorem}{Theorem}[section]
\newtheorem{lemma}[theorem]{Lemma}
\newtheorem{prop}[theorem]{Proposition}

\newtheorem{cor}[theorem]{Corollary}

\theoremstyle{remark}

\newtheorem*{example}{Example}

\newtheorem{rem}[theorem]{Remark}
\newtheorem{defn}[theorem]{Definition}
\newtheorem{assump}[theorem]{Assumption}
\newtheorem{nrem}[theorem]{Notational Remark}

\newcommand{\R} {\mathbb{R}}

\newcommand{\N} {\mathbb{N}}

\newcommand{\E} {\mathbb{E}}

\newcommand{\p} {\mathbb{P}}

\DeclareMathOperator{\Tr}{Tr}
\DeclareMathOperator{\Var}{Var}
\DeclareMathOperator{\Cov}{Cov}

\DeclareMathOperator{\sech}{sech}

\newcommand{\caP}{{\mathcal P}}

\newcommand{\caL}{{\mathcal L}}
\newcommand{\caN}{{\mathcal N}}
\newcommand{\caO}{{\mathcal O}}

\newcommand{\bsH}{{\boldsymbol H}}

\newcommand{\bsx}{{\boldsymbol x}}

\newcommand{\wt}{\widetilde}
\newcommand{\ol}{\overline}

\newcommand{\beq}{ \begin{equation} }
\newcommand{\eeq}{ \end{equation} }

\newcommand{\dd}{\mathrm{d}}
\newcommand{\ii}{\mathrm{i}}

\renewcommand{\P}{\mathbb{P}}

\newcommand{\SNR}{\omega}
\newcommand{\tM}{\widetilde{M}}

\numberwithin{equation}{section} %\numberwithin{lem}{section}
\title{Asymptotic Normality of Log Likelihood Ratio and Fundamental Limit of the Weak Detection for Spiked Wigner Matrices}

\author{
	Hye Won Chung \footnote{School of Electrical Engineering, KAIST, Daejeon, 34141, Korea
		\newline email: \texttt{hwchung@kaist.ac.kr}}, \;
	Jiho Lee \footnote{Korea Science Academy of KAIST, Busan, 10547, Korea
		\newline email: \texttt{efidiaf@ksa.hs.kr}},	\;
	and Ji Oon Lee\footnote{Department of Mathematical Sciences, KAIST, Daejeon, 34141, Korea
		\newline email: \texttt{jioon.lee@kaist.edu}}}

\begin{document}

\maketitle

\begin{abstract}
We consider the problem of detecting the presence of a signal in a rank-one spiked Wigner model. For general non-Gaussian noise, assuming that the signal is drawn from the Rademacher prior, we prove that the log likelihood ratio (LR) of the spiked model against the null model converges to a Gaussian when the signal-to-noise ratio is below a certain threshold. The threshold is optimal in the sense that the reliable detection is possible by a transformed principal component analysis (PCA) above it. From the mean and the variance of the limiting Gaussian for the log-LR, we compute the limit of the sum of the Type-I error and the Type-II error of the likelihood ratio test. We also prove similar results for a rank-one spiked IID model where the noise is asymmetric but the signal is symmetric.
\end{abstract}

\section{Introduction} \label{sec:intro}

One of the most fundamental questions in statistics and data science is to detect a signal from noisy data. The simplest model for the `signal-plus-noise' data is the spiked Wigner matrix, where the signal is a vector and the noise is a symmetric random matrix. In this model, the data matrix is of the form
\[
	M = H+\sqrt{\lambda} \bsx \bsx^T,
\]
where the signal $\bsx$ is an $N$-dimensional vector and the noise matrix $H$ is an $N \times N$ Wigner matrix. (See Definition \ref{defn:wigner} and Assumption \ref{assump:wigner} for a precise definition.) The parameter $\lambda$ is the signal-to-noise ratio (SNR).

Signal detection/recovery problems for the spiked Wigner matrix have been widely studied in the past decades. One of the most natural ways of detecting the signal is to analyze the eigenvalues of the data matrix, especially its largest eigenvalues, which is called the principal component analysis (PCA). With the normalization $\E[|H_{ij}|^2]=N^{-1}$ for $i \neq j$ and $\| \bsx \| = 1$, it is well-known in random matrix theory that the bulk of the spectrum is supported on $[-2, 2]$ and the largest eigenvalue exhibits the following sharp phase transition. If $\lambda > 1$, the largest eigenvalue of $M$ separates from the bulk, and hence the signal can be reliably detected by PCA. On the other hand, if $\lambda < 1$, the largest eigenvalue converges to $2$ and its fluctuation does not depend on $\lambda$; thus, the largest eigenvalue carries no information on the signal and PCA fails. 

The spectral threshold $\lambda_c=1$ for the PCA, and it is optimal when the noise is Gaussian in the sense that reliable detection is impossible below the spectral threshold \cite{Perry2018}. If the noise is non-Gaussian, however, the conventional PCA is sub-optimal and can be improved by transforming the data entrywise \cite{NIPS_Barbier,Perry2018}. In this case, with an optimal entrywise transformation, under suitable conditions, the threshold is lowered to $1/F$, where $F$ is the Fisher information of the (off-diagonal) noise distribution, which is strictly larger than $1$ if the noise is non-Gaussian. (See \eqref{eq:FG} for its definition.)

Below the threshold $\lambda_c = 1/F$ for the non-Gaussian noise case, reliable detection is impossible \cite{Perry2018}, and it is only possible to consider hypothesis testing between the null model $\lambda=0$ and the alternative $\lambda>0$, which is called the weak detection. The fundamental limit of the weak detection can be proved by analyzing the performance of the likelihood ratio (LR) test, which minimizes the sum of the Type-I and Type-II errors by Neyman--Pearson lemma. For Gaussian noise, the log-LR coincides with the free energy of the Sherrington--Kirkpatrick (SK) model of spin glass. The limiting fluctuation of the free energy of the SK model is given by a Gaussian in the high-temperature case, and the result is directly applicable to the proof of asymptotic normality of the log-LR and the limiting error of the LR test.

While the performance of the LR test is of fundamental importance, to our best knowledge, no results have been proved when the noise is non-Gaussian. The main difficulty in the analysis of the LR test is that the connection between the log-LR and the free energy of the SK model is absent when the noise is non-Gaussian; though the free energy of the SK model converges to a normal distribution with non-Gaussian interaction for some cases \cite{ALR87,Baik-Lee2016}, it does not prove the Gaussian convergence of the log-LR. 

In this paper, we focus on the likelihood ratio between the null hypothesis $\bsH_0 : \lambda=0$ and the alternative hypothesis $\bsH_1 : \lambda = \SNR > 0$ for the subcritical case $\SNR < 1/F$ under the assumption that the distribution of the signal, called the prior, is Rademacher, i.e., the $x_i$'s are i.i.d. with $\p(\sqrt{N}x_i=1) = \p(\sqrt{N}x_i=-1) = 1/2$. Our goals are to prove that the asymptotic normality holds even with non-Gaussian noise under mild assumption on the noise and to compute the exact mean and the variance for the limiting Gaussian that enables us to find the limiting error of the LR test.

An interesting variant of the spiked Wigner matrix is the spiked IID matrix, where the entries of the noise matrix are independent and identically distributed without symmetry constraint. Such a model is natural when the $(i, j)$-entry and $(j, i)$-entry are sampled separately. While it seems natural to apply the singular value decomposition (SVD) for this case, it is in fact better to symmetrize the matrix first and then apply PCA to exploit the symmetry of the signal; for example, with the Gaussian noise, the threshold for the former $\lambda_c=1$, whereas $\lambda_c=1/2$ for the latter. The idea of symmetrization can even be combined with the entrywise transformation in \cite{Perry2018} if the noise is non-Gaussian to prove that reliable detection is possible if $\lambda > 1/(2F)$. However, to our best knowledge, no results are known for the subcritical case $\lambda < 1/(2F)$.

Our goal for the spiked IID model is to prove the results that correspond to those for the spiked Wigner matrices, including the asymptotic normality of the log-LR for the spiked IID matrices and the precise formulas for the mean and the variance for the limiting Gaussian, for $\lambda < 1/(2F)$. We remark that a consequence of the asymptotic normality of the log-LR and Le Cam's first lemma is the impossibility of reliable detection for $\lambda < 1/(2F)$, which establishes that the threshold for the reliable detection for the spiked IID matrices is given by $\lambda_c = 1/(2F)$ for the non-Gaussian noise case.

\subsection{Main contribution on spiked Wigner matrices} \label{subsec:intro_Wigner}

For the spiked Wigner matrices,
\begin{itemize}
\item (Asymptotic normality - Theorem \ref{thm:main}) we prove that the log-LR converges to a Gaussian, and
\item (Optimal error - Corollary \ref{cor:main}) we compute the limit of the sum of the Type-I and Type-II errors of the LR test.
\end{itemize}

Let $\p_0$ and $\p_1$ be the probability distributions of a spiked Wigner matrix $M$ under $\bsH_0$ and $\bsH_1$, respectively. If the noise is Gaussian, from the normalization in Definition \ref{defn:wigner} and Assumption \ref{assump:wigner}, the likelihood ratio of $\p_1$ with respect to $\p_0$ is given by
\beq \begin{split} \label{eq:Gaussian_LR}
	&\caL(M;\SNR) = \frac{\dd \p_1}{\dd \p_0} \\
	&:= \frac{1}{2^N} \sum_{\bsx} \prod_{i<j} \frac{\exp (-\frac{N}{2} (\sqrt{N} M_{ij}-\sqrt{\SNR N} x_i x_j)^2)}{\exp (-\frac{N}{2} (\sqrt{N} M_{ij})^2)} \prod_k \frac{\exp (-\frac{N}{2}(\sqrt{N} M_{kk}-\sqrt{\SNR N} x_k^2)^2)}{\exp (-\frac{N}{2}(\sqrt{N} M_{kk})^2)} \\
	&= \frac{1}{2^N} \sum_{\bsx} \prod_{i<j} \exp \left(\sqrt{\SNR} N^2 M_{ij} x_i x_j - \frac{\SNR}{2} \right) \prod_k \exp \left(N \sqrt{\SNR} M_{kk}- \frac{\SNR}{2} \right),
\end{split} \eeq
where we used that $x_i^2=N^{-1}$ for all $i$. The random off-diagonal term in \eqref{eq:Gaussian_LR},
\[
	\frac{1}{2^N} \sum_{\bsx} \prod_{i<j} \exp \left(\sqrt{\SNR} N^2 M_{ij} x_i x_j \right) = \frac{1}{2^N} \sum_{\bsx} \exp \left(\sqrt{\SNR} N^2 \sum_{i<j} M_{ij} x_i x_j \right),
\]
coincides with the free energy of the Sherrington--Kirkpatrick model for which the asymptotic normality is known. (The diagonal part is independent from the off-diagonal part and its asymptotic (log)-normality can be easily proved by the CLT.)

If the noise is non-Gaussian, however, then we instead have
\beq \label{eq:non-Gaussian_LR}
	\caL(M;\SNR) = \frac{1}{2^N} \sum_{\bsx} \prod_{i<j} \frac{p(\sqrt{N} M_{ij}-\sqrt{\SNR N} x_i x_j)}{p(\sqrt{N} M_{ij})} \prod_k \frac{p_d(\sqrt{N} M_{kk}-\sqrt{\SNR N} x_k^2)}{p_d(\sqrt{N} M_{kk})}\,,
\eeq
where $p$ and $p_d$ are the densities of the (normalized) off-diagonal terms and the diagonal terms, respectively. Applying the Taylor expansion, 
\[
	\frac{p(\sqrt{N} M_{ij}-\sqrt{\SNR N} x_i x_j)}{p(\sqrt{N} M_{ij})} \simeq 1 -\sqrt{\SNR N} \, \frac{p'(\sqrt{N} M_{ij})}{p(\sqrt{N} M_{ij})} x_i x_j + \frac{\SNR}{2N} \frac{p''(\sqrt{N} M_{ij})}{p(\sqrt{N} M_{ij})}\,,
\]
and after comparing it to the Taylor expansion of the exponential function, it is possible to rewrite the leading order terms of the off-diagonal part of the LR as
\beq \label{eq:heuristic_LR}
	\frac{1}{2^N} \sum_{\bsx} \exp \left( \sum_{i<j} (A_{ij} x_i x_j + B_{ij}) \right),
\eeq
where $A_{ij}$ and $B_{ij}$ depend only on $p(\sqrt{N} M_{ij})$ and its derivatives. (See \eqref{eq:LR_formula}-\eqref{eq:ABC} for a precise definition.) Since $A_{ij}$ and $B_{ij}$ are not independent, the asymptotic normality of \eqref{eq:heuristic_LR} is not obvious, and in fact, requires hard analysis.

The main technical difficulty in the analysis of the log-LR arises from that the mechanism for the Gaussian convergence of the terms containing $A_{ij}$, which we call the spin glass term, and those with $B_{ij}$, which we call the CLT term, are markedly different. While the latter can be readily proved by the CLT, the former is a result from the theory of spin glass based on the combinatorial analysis for the graph structure involving $A_{ij}$'s.
To overcome the difficulty, adopting the strategy of \cite{ALR87}, we decompose the spin glass term, $2^{-N} \sum_{\bsx} \exp ( \sum_{i<j} (A_{ij} x_i x_j) )$, into two parts, which are asymptotically orthogonal to each other, conditional on $(B_{ij})$. (See Proposition \ref{prop:Z_limit} for more detail.) We can then verify the part that depends on $(B_{ij})$ and express the dependence by a conditional expectation on $(B_{ij})$. Combining the $B$-dependent part of the spin glass term with the CLT term, we can prove the Gaussian convergence with a precise formula for the mean and the variance of the limiting distribution.

From the asymptotic normality, it is immediate to obtain the sum of the Type-I and Type-II errors of the LR test. The limiting error in Corollary \ref{cor:main} converges to $0$ as $\lambda \nearrow (1/F)$, which shows that the detection threshold is given by $\lambda_c = 1/F$. While the limiting error of the LR test for the case $\lambda > 1/F$ also converges to $0$, which can be indirectly deduced from the fact that the transformed PCA in \cite{Perry2018} can reliably detect the signal, our result is not directly applicable to this case; the asymptotic normality holds only in the regime $\lambda < 1/F$ and it is expected the asymptotic behavior of the log-LR changes if $\lambda > 1/F$, which corresponds to the low temperature regime in the theory of spin glass.

The Rademacher prior we consider in this work is one of the most natural models for the rank-one signal-plus-noise data, and it also corresponds to the most fundamental spin configuration in the SK model. Our method relies on this special structure of the spike at least in the technical level. We believe that our main result would change with different priors, e.g., the spherical prior instead of the Rademacher prior, but the change is restricted to a subleading term (proportional to $\SNR^2$) in the mean and the variance of the limiting Gaussian. Note that such a change would not occur with the Gaussian noise where the quadratic term (of $\SNR$) vanishes, which can be made rigorously since the asymptotic behavior of the log-LR with the spherical prior coincides that with many i.i.d. priors including the Rademacher prior \cite{AlaouiJordan2018}. See the discussion at the end of Section \ref{subsec:main} and also the numerical experiments in Section \ref{sec:ex} for details.

\subsection{Main contribution on spiked IID matrices} \label{subsec:intro_iid}

For the spiked IID matrices, we also prove the asymptotic normality of the log-LR (Theorem \ref{thm:iid}) and compute the limit of the sum of the Type-I and Type-II errors of the LR test (Corollary \ref{cor:iid}).

Suppose that the data matrix $Y$ is of the form
$Y = X+\sqrt{\lambda} \bsx \bsx^T$,
where the signal $\bsx$ is an $N$-dimensional vector and the noise $X$ is an $N \times N$ IID matrix. (See Definition \ref{defn:iid} for a precise definition.) If the noise is Gaussian, similar to the spiked Wigner case in \eqref{eq:Gaussian_LR}, we find that the likelihood ratio
\beq \begin{split} \label{eq:Gaussian_LR_iid}
	\caL(Y;\SNR) &:= \frac{\dd \p_1}{\dd \p_0} := \frac{1}{2^N} \sum_{\bsx} \prod_{i, j=1}^N \frac{\exp (-\frac{N}{2} (\sqrt{N} Y_{ij}-\sqrt{\SNR N} x_i x_j)^2)}{\exp (-\frac{N}{2} (\sqrt{N} Y_{ij})^2)} \\
	&= \frac{1}{2^N} \sum_{\bsx} \prod_{i<j} \exp \left(\sqrt{\SNR} N^2 (Y_{ij}+Y_{ji}) x_i x_j - \SNR \right) \prod_k \exp \left(N \sqrt{\SNR} Y_{kk}- \frac{\SNR}{2} \right).
\end{split} \eeq
Since $(Y_{ij}+Y_{ji})$ is again a Gaussian random variable whose variance is $2N^{-1}$, under $\bsH_0$, the exponent in the off-diagonal term $(\sqrt{\SNR} N^2 (Y_{ij}+Y_{ji}) x_i x_j - \SNR )$ is equal in distribution to $(\sqrt{2\SNR} N^2 Y_{ij} x_i x_j - \SNR )$. Thus, when compared to the exponent in the off-diagonal term in \eqref{eq:Gaussian_LR}, we find that the SNR $\SNR$ is effectively doubled, and in particular, reliable detection is impossible if $\SNR < 1/2$. Note that in this case it is possible to reliably detect the signal if $\SNR > 1/2$ by PCA for the symmetrized matrix $(Y+Y^T)/\sqrt{2}$, whereas the largest singular value of $M$ is separated from other singular values only when $\SNR > 1$ (see, e.g., Section 3.1 of \cite{benaych2012singular}).

If the noise is non-Gaussian, however, we face the same issue as in the spiked Wigner matrices; the log-LR can be decomposed into the spin glass term and the CLT term as in \eqref{eq:heuristic_LR}, which are not independent. The issue can be resolved by adapting the analysis for the spiked Wigner matrices, where we need to consider $A_{ij} + A_{ji}$ and $B_{ij} + B_{ji}$ in place of $A_{ij}$ and $B_{ij}$. It is then possible to prove the asymptotic normality of the log-LR with precise formulas for the mean and the variance.

The results in Theorem \ref{thm:iid} and Corollary \ref{cor:iid} show that reliable detection is impossible if $\lambda < 1/(2F)$. On the other hand, if we transform the data matrix entrywise and then symmetrize the transformed matrix, the largest eigenvalue of the resulting matrix becomes an outlier when $\lambda > 1/(2F)$, and hence reliable detection is possible by PCA. Thus, we can conclude that the threshold $\lambda_c = 1/(2F)$ for the reliable detection. See the discussion at the end of Section \ref{subsec:asymmetric} for details.

\subsection{Related works}

The phase transition of the largest eigenvalue of the spiked Wigner matrix has been extensively studied in random matrix theory. It is known as the BBP transition after the seminal work by Baik, Ben Arous, and P\'ech\'e \cite{BBP2005} for spiked (complex) Wishart ensembles. Corresponding results for the spiked Wigner matrix were proved under various assumptions \cite{Peche2006,FeralPeche2007,CapitaineDonatiFeral2009,Raj2011}. For the behavior of the eigenvector associated with the largest eigenvalue of the spiked Wigner matrix, we refer to \cite{Raj2011}.

The impossibility of the reliable detection for the case $\lambda <1$ with the Gaussian noise was considered by Montanari, Reichman, and Zeitouni \cite{Montanari2017} and Johnstone and Onatski \cite{johnstone2020testing}. The optimality of the PCA for the Gaussian noise, the sub-optimality of the PCA and the optimality of the transformed PCA for the non-Gaussian noise was proved by Perry, Wein, Bandeira, and Moitra \cite{Perry2018} under mild assumptions on the prior. We remark that the threshold $\lambda_c$ may not be exactly $1$ depending on the prior \cite{AlaouiJordan2018}.

The analysis on the LR test and its fundamental limit were studied by El Alaoui, Krzakala, and Jordan \cite{AlaouiJordan2018} under the assumption that the signal is an i.i.d. random vector and the noise is Gaussian. We remark that there exists a non-LR test based on the linear spectral statistics of the data matrix, which is optimal if the noise is Gaussian \cite{chung2019weak}.

The detection problem for the spiked rectangular model, introduced by Johnstone \cite{Johnstone2001}, has also been widely studied. We refer to \cite{Onatski2013,Onatski2014,el2018detection,pmlr-v139-jung21a,Dobriban2017} for the weak detection under Gaussian noise, including the optimal error of the hypothesis test, and \cite{Perry2018,pmlr-v139-jung21a} for the sub-optimality of the PCA under non-Gaussian noise. For the spiked IID matrices and their applications, we refer to \cite{chen2021asymmetry} and references therein.

The SK model is one of the most fundamental models of spin glass. The model was introduced by Sherrington and Kirkpatrick \cite{sherrington1975solvable} as a mean-field version of the Edwards--Anderson model \cite{edwards1975theory}. The (deterministic) limit of the free energy of the SK model was first predicted by Parisi \cite{parisi1980sequence}, and it was later rigorously proved by Talagrand \cite{talagrand2006parisi}. The fluctuation of the free energy for the SK model was studied by Aizenman, Lebowitz, and Ruelle \cite{ALR87}; they showed the Gaussian convergence of the fluctuation of the free energy in the high temperature case, which also holds with non-Gaussian noise. To the best of our best knowledge, the fluctuation in the low temperature case, including the zero-temperature case, still remains open.

\subsection{Organization of the paper}

The rest of the paper is organized as follows. In Section \ref{sec:prelim}, we precisely define the model and state the main results. In Section \ref{sec:ex}, we present some examples and conduct numerical experiments to compare the outcomes with the theoretical results. In Sections \ref{sec:proof} and \ref{sec:indep}, we prove our main results on the spiked Wigner matrices. We conclude the paper in Section \ref{sec:conclusion} by summarizing our results and proposing future research directions. Some details of the proof, including the proof for the results on the spiked IID matrices, can be found in Appendix.

\section{Main results} \label{sec:prelim}

\subsection{Definition of the model} \label{subsec:model}

We begin by precisely defining the model we consider in this paper. We suppose that the noise matrix is a Wigner matrix, which is defined as follows:

\begin{defn}[Wigner matrix] \label{defn:wigner}
	We say an $N \times N$ random matrix $H = (H_{ij})$ is a Wigner matrix if $H$ is a real symmetric matrix and $H_{ij}$ ($1\leq i\leq j \leq N$) are independent centered real random variables satisfying
	\begin{itemize}
	\item (Normalization) $N \E[H_{ij}^2]=1$ for all $1 \leq i < j \leq N$ and $N \E[H_{ii}^2]=w_2$ for some $w_2$, independent of $N$,
	\item (Finite moment) for any $D>0$, there exists a constant $C_D$, independent of $N$, such that $N^{\frac{D}{2}} \E[H_{ij}^D] \leq C_D$ for all $1 \leq i \leq j \leq N$.
	\end{itemize}
\end{defn}

For the data matrix $M$, we assume that it is a spiked Wigner matrix, i.e., $M = H+\sqrt{\lambda} \bsx \bsx^T$, and also additionally assume the following:

\begin{assump} \label{assump:wigner}
	For the spike $\bsx$, we assume that the normalized entries $\sqrt{N} x_i$ are i.i.d. random variables with Rademacher distribution. 
	
	For the noise matrix $H$, let $p$ and $p_d$ be the densities of the normalized off-diagonal entries $\sqrt{N} H_{ij}$ and the normalized diagonal entries $\sqrt{N} H_{ii}$, respectively. We assume the following:
	\begin{itemize}
	\item The density functions $p$ and $p_d$ are smooth, positive everywhere, and symmetric (about $0$).
	\item The functions $p$, $p_d$, and their all derivatives vanish at infinity.
	\item The functions $p^{(s)}/p$ and $p_d^{(s)}/p_d$ are polynomially bounded, i.e., for any positive integer $s$ there exist constants $C_s, r_s >0$, independent of $N$, such that $|p^{(s)}(x)/p(x)|, |p_d^{(s)}(x)/p_d(x)| \leq C_s |x|^{r_s}$ uniformly on $x$. (Here, $p^{(s)}$ and $p_d^{(s)}$ are the $s$-th derivatives of $p$ and $p_d$, respectively.)
	\end{itemize}
\end{assump}

Note that we assume the off-diagonal entries are identically distributed. While it is possible to consider a spiked Wigner matrix with non-identically distributed noise matrix by assuming certain conditions on their density functions, we refrain from it to avoid unnecessary complications. Similarly, we remark that some conditions in Definition \ref{defn:wigner} and Assumption \ref{assump:wigner} are due to technical reasons, especially the finite moment condition and the symmetric condition for the density of the noise, and we believe that it is possible to relax the conditions. However, for the sake of brevity, we do not attempt to relax them further in this paper.

\subsection{Spiked Wigner matrices} \label{subsec:main}

Recall that the SNR $\SNR$ for the alternative $\bsH_1$ is fixed. Our main result is the following asymptotic normality for the log-LR statistic when testing between $\bsH_0 : \lambda=0$ versus $\bsH_1 : \lambda = \SNR > 0$.

\begin{theorem} \label{thm:main}
	Suppose that $M$ is a spiked Wigner matrix satisfying Assumption \ref{assump:wigner}. Define
\beq \label{eq:FG}
	F := \int_{-\infty}^{\infty} \frac{(p'(x))^2}{p(x)} \dd x, \quad F_d := \int_{-\infty}^{\infty} \frac{(p_d'(x))^2}{p_d(x)} \dd x, \quad G := \int_{-\infty}^{\infty} \frac{(p''(x))^2}{p(x)} \dd x\,.
\eeq
	If $\omega F < 1$, then the log likelihood ratio $\log \caL(M;\lambda)$ of $\caP_1$ with respect to $\caP_0$ under $\bsH_0$ (defined in \eqref{eq:non-Gaussian_LR}) converges in distribution to $\caN(-\rho, 2\rho)$ as $N \to \infty$, where
\beq \label{eq:rho}
	\rho := -\frac{1}{4} \left( \log (1-\SNR F) + \SNR (F-2F_d) + \frac{\SNR^2}{4} (2F^2 - G) \right).
\eeq
\end{theorem}

\begin{rem} \label{rem:main}
Applying Le Cam's third lemma, we also find from Theorem \ref{thm:main} that the log likelihood ratio $\log \caL(M;\lambda)$ under $\bsH_1$ converges in distribution to $\caN(\rho, 2\rho)$.
\end{rem}

If the off-diagonal density $p$ is Gaussian, we find that $F=1$, $G=2$, and the result in Theorem \ref{thm:main} coincides with Theorem 2 (and Remark below it) of \cite{AlaouiJordan2018}, which states that $\log \caL(M;\lambda)$ converges in distribution to $\caN(-\rho', 2\rho')$ where
\[
	\rho' = -\frac{1}{4} \big( \log (1-\SNR) + \SNR (1-2F_d) \big).
\]
From the comparison with the Gaussian case, ignoring the term $(\SNR^2/4) (2F^2 - G)$ in \eqref{eq:rho}, the result in Theorem \ref{thm:main} heuristically shows that the effective SNR is $\SNR F$ when the noise is non-Gaussian. 

The change of the effective SNR with non-Gaussian noise can be understood as follows: under $\bsH_1$, if we apply a function $q$ to a normalized entry $\sqrt{N} M_{ij}$, then by the Taylor expansions,
\beq \label{eq:q}
	q(\sqrt{N} M_{ij}) = q(\sqrt{N} H_{ij} + \sqrt{\SNR N} x_i x_j) \approx q(\sqrt{N} H_{ij}) + \sqrt{\SNR N} q'(\sqrt{N} H_{ij}) x_i x_j.
\eeq
Approximating the right-hand side further by
\[ \begin{split}
	q(\sqrt{N} H_{ij}) + \sqrt{\SNR N} q'(\sqrt{N} H_{ij}) x_i x_j & \approx q(\sqrt{N} H_{ij}) + \sqrt{\SNR N} \E[q'(\sqrt{N} H_{ij})] x_i x_j \\
	&= \sqrt{N} \left( \frac{q(\sqrt{N} H_{ij})}{\sqrt{N}} + \sqrt{\SNR} \E[q'(\sqrt{N} H_{ij})] x_i x_j \right),
\end{split} \]
we find that the matrix obtained by transforming $M$ entrywise via $q$ is of the form $\sqrt{N} (Q + \sqrt{\SNR'} \bsx \bsx^T)$, with $Q_{ij} =q(\sqrt{N} H_{ij})/\sqrt{N}$ and $\sqrt{\SNR'} = \sqrt{\SNR} \E[q'(\sqrt{N} H_{ij})]$, which shows that the SNR is effectively changed. The approximation can be made rigorous \cite{Perry2018}, and after optimizing $q$, it can be shown that effective SNR is $\SNR F$, which is independent of the prior under mild assumptions. We remark that the optimal $q$ is given by $-p'/p$, and it is unique up to a constant factor.

An immediate consequence of Theorem \ref{thm:main} is the following estimate for the error probability of the LR test.

\begin{cor} \label{cor:main}
	Suppose that the assumptions of Theorem \ref{thm:main} hold. Then, the error probability of the likelihood ratio test 
\[
	\mathrm{err}(\SNR) := \p( \mathrm{Reject}\, \bsH_0 | \bsH_0 ) + \p( \mathrm{Do} \,\,  \mathrm{not} \,\,  \mathrm{reject}\, \bsH_0 | \bsH_1 )
\]
converges to
\[
	\mathrm{erfc} \left( \frac{\sqrt{\rho}}{2} \right) = \frac{2}{\sqrt{\pi}} \int_{\frac{\sqrt{\rho}}{2}}^{\infty} e^{-x^2} \dd x
\]
as $N \to \infty$, where $\rho$ is as defined in \eqref{eq:rho}.
\end{cor}

The proof of Corollary \ref{cor:main} from Theorem \ref{thm:main} is standard; we refer to the proof of Theorem 2 of \cite{chung2019weak}. 

Note that the error $\mathrm{err}(\SNR)$ converges to $0$ as $\SNR F \nearrow 1$, which is consistent with the discussion below Theorem \ref{thm:main} that the effective SNR is $\SNR F$.

\begin{rem} \label{rem:main_cor}
Since the likelihood ratio test minimizes the sum of the Type-I error and the Type-II error by the Neyman--Pearson lemma, we find that for any hypothesis test $\bsH_0$ versus $\bsH_1$, $\mathrm{err}(\SNR)$ is asymptotically bounded above by $\mathrm{erfc}(\sqrt{\rho}/2)$.
\end{rem}

The most non-trivial part in Theorem \ref{thm:main} and Corollary \ref{cor:main} is the third term of $\rho$ in \eqref{eq:rho},
\beq \label{eq:error_third_term}
	\frac{\SNR^2}{4} (2F^2 - G).
\eeq
We conjecture that this term is due to the Rademacher prior and it would change if we assume a different prior. Our heuristic argument for it is as follows. In \cite{chung2019weak}, an algorithm was proposed for hypothesis testing based on the linear spectral statistics (LSS) of the data matrix $M$, which is optimal if the noise is Gaussian; the LSS of the matrix $M$ whose eigenvalues are $\lambda_1 \geq \lambda_2 \geq \dots \geq \lambda_N$ is
\[
	\sum_{i=1}^N f(\lambda_i)
\]
for a function $f$ continuous on an open neighborhood of $[-2, 2]$. If the noise is non-Gaussian, the algorithm in \cite{chung2019weak} can be improved by pre-transforming the data matrix entrywise as in \eqref{eq:q}. The error of the improved algorithm is
\beq \label{eq:wt_rho}
	\mathrm{erfc} \left( \sqrt{\wt \rho}/2 \right),
\eeq
where
\[
	\wt \rho = -\frac{1}{4} \left( \log (1-\SNR F) + \SNR (F-2F_d) + \frac{\SNR^2}{4} (2F^2 - \wt G) \right)
\]
with
\[
	\wt G = \left( \int_{-\infty}^{\infty} \frac{(p'(x))^2 p''(x)}{p(x)} \dd x \right)^2 \Big/ \left( \frac{3}{2} \int_{-\infty}^{\infty} \frac{(p'(x))^2 p''(x)}{p(x)} \dd x - F^2 \right).
\]
The error \eqref{eq:wt_rho} differs from the error in Corollary \ref{cor:main} only in the term involving $\SNR^2$. The hypothesis test based on the eigenvalues are spherical in nature, since it does not depend on the eigenvectors, or specific directions. It means that the prior naturally associated with the LSS-based test is the spherical prior, which is the uniform distribution on the unit sphere. We thus conjecture that the error term in \eqref{eq:error_third_term} depends on the prior of the model.

\subsection{Spiked IID matrices} \label{subsec:asymmetric}

The results on the spiked Wigner matrices in Theorem \ref{thm:main} can be naturally extended to the spiked IID matrices defined as follows:

\begin{defn}[IID matrix] \label{defn:iid}
	We say an $N \times N$ real random matrix $X = (X_{ij})$ is an IID matrix if $X_{ij}$ ($1\leq i, j \leq N$) are independent and identically distributed centered real random variables satisfying that (i) $N \E[X_{ij}^2]=1$ and (ii) for any $D>0$, there exists a constant $C_D$, independent of $N$, such that $N^{\frac{D}{2}} \E[X_{ij}^D] \leq C_D$.
\end{defn}

For a spiked IID matrix of the form $Y = X+\sqrt{\lambda} \bsx \bsx^T$, we have following result.

\begin{theorem} \label{thm:iid}
	Suppose that $Y = X+\sqrt{\lambda} \bsx \bsx^T$ is a spiked IID matrix and Assumption \ref{assump:wigner} holds for $Y$ where $p$ is the density of the normalized entry $\sqrt{N} X_{ij}$. Let
\beq
	F := \int_{-\infty}^{\infty} \frac{(p'(x))^2}{p(x)} \dd x, \quad G := \int_{-\infty}^{\infty} \frac{(p''(x))^2}{p(x)} \dd x\,.
\eeq
	If $2 \omega F < 1$, then the log likelihood ratio $\log \caL(Y;\lambda)$ of $\caP_1$ with respect to $\caP_0$ under $\bsH_0$, defined by
\beq
	\caL(Y;\SNR) := \frac{1}{2^N} \sum_{\bsx} \prod_{i, j=1}^N \frac{p(\sqrt{N} Y_{ij}-\sqrt{\SNR N} x_i x_j)}{p(\sqrt{N} Y_{ij})}
\eeq
converges in distribution to $\caN(-\rho^*, 2\rho^*)$ as $N \to \infty$, where
\beq \label{eq:rho_star}
	\rho^* := -\frac{1}{4} \left( \log (1- 2 \SNR F) + \frac{\SNR^2}{2} (2F^2 - G) \right).
\eeq
\end{theorem}

See Appendix \ref{sec:iid_proof} for the proof of Theorem \ref{thm:iid}. From Theorem \ref{thm:iid}, we can obtain the following estimate for the error probability of the LR test.

\begin{cor} \label{cor:iid}
	Suppose that the assumptions of Theorem \ref{thm:iid} hold. Then, the error probability of the likelihood ratio test 
\[
	\mathrm{err}(\SNR) := \p( \mathrm{Reject}\, \bsH_0 | \bsH_0 ) + \p( \mathrm{Do} \,\,  \mathrm{not} \,\,  \mathrm{reject}\, \bsH_0 | \bsH_1 )
\]
converges to
\[
	\mathrm{erfc} \left( \frac{\sqrt{\rho^*}}{2} \right) = \frac{2}{\sqrt{\pi}} \int_{\frac{\sqrt{\rho^*}}{2}}^{\infty} e^{-x^2} \dd x
\]
as $N \to \infty$, where $\rho^*$ is as defined in \eqref{eq:rho_star}.
\end{cor}

We omit the proof of Corollary \ref{cor:iid}; see the proof of Theorem 2 of \cite{chung2019weak}.

Comparing Theorem \ref{thm:iid} with Theorem \ref{thm:main}, we find that the effective SNR is $2\SNR F$, which is doubled from the effective $\SNR F$ for the spiked Wigner matrices. The doubling of the effective SNR can also be heuristically checked from the LR as follows. By definition,
\beq \begin{split}
	\caL(Y;\SNR) = \frac{1}{2^N} \sum_{\bsx} \prod_{i \neq j} \frac{p(\sqrt{N} Y_{ij}-\sqrt{\SNR N} x_i x_j)}{p(\sqrt{N} Y_{ij})} \prod_k \frac{p(\sqrt{N} X_{kk}-\sqrt{\SNR N} x_k^2)}{p(\sqrt{N} X_{kk})}\,,
\end{split} \eeq
The off-diagonal term can be further decomposed into the upper-diagonal term and the lower-diagonal term, i.e.,
\[
	\prod_{i \neq j} \frac{p(\sqrt{N} Y_{ij}-\sqrt{\SNR N} x_i x_j)}{p(\sqrt{N} Y_{ij})} = \prod_{i < j} \frac{p(\sqrt{N} Y_{ij}-\sqrt{\SNR N} x_i x_j)}{p(\sqrt{N} Y_{ij})} \prod_{i > j} \frac{p(\sqrt{N} Y_{ij}-\sqrt{\SNR N} x_i x_j)}{p(\sqrt{N} Y_{ij})}.
\]
The upper-diagonal term and the lower-diagonal term are almost independent, which effectively makes the LR becomes squared and the log-LR doubled. However, these terms are not exactly independent due to the spike $x_i x_j$, and this is why the quadratic term $(\SNR^2/2) (2F^2 - G)$ in $\rho^*$ is not the same as the four times the quadratic term $(\SNR^2/4) (2F^2 - G)$, which is what one would get by changing $\SNR$ to $2\SNR$ in \eqref{eq:rho}. Note that the linear term $\SNR (F-2F_d)$ vanishes in \eqref{eq:rho_star} due to the assumption that $F_d = F$, since the term $\SNR F$, arising from the off-diagonal terms, gets doubled for the spiked IID matrices but $2\SNR F_d$ does not as it originates from the diagonal terms.

Theorem \ref{thm:iid} implies the impossibility of the reliable detection when $\lambda < 1/(2F)$. To conclude that the threshold $\lambda_c = 1/(2F)$ for the reliable detection, we can consider the following variant of the PCA. We first apply the function $q$ entrywise to the data matrix as in \eqref{eq:q}. After normalizing the transformed matrix by dividing it by $\sqrt{N}$, it is approximately a spiked IID matrix. In the next step, we symmetrize the (normalized) transformed matrix, i.e., we compute
\[
	\wt Y_{ij} = \frac{q(\sqrt{N} Y_{ij}) + q(\sqrt{N} Y_{ji})}{\sqrt{2N}}.
\] 
From the Taylor expansion, we find that
\[ \begin{split}
	\frac{q(\sqrt{N} Y_{ij}) + q(\sqrt{N} Y_{ji})}{\sqrt{2N}} \approx \frac{q(\sqrt{N} X_{ij}) + q(\sqrt{N} X_{ji})}{\sqrt{2N}} + \sqrt{2\SNR} \E[q'(\sqrt{N} X_{ij})] x_i x_j,
\end{split} \]
which is approximately a spiked Wigner matrix. Following the proof in \cite{Perry2018}, it is then immediate to prove that the reliable detection is possible by applying the PCA to $\wt Y$ if $\lambda_c > 1/(2F)$.

\section{Examples and numerical experiments} \label{sec:ex}

In this section, we consider a spiked Wigner model with non-Gaussian noise and conduct numerical experiments to compare the error from the simulation with the theoretical error in Corollary \ref{cor:main}.

Suppose that the density of the noise matrix is
\[
	p(x) = p_d(x) = \frac{\sech(\pi x/2)}{2} = \frac{1}{e^{\pi x /2} + e^{-\pi x/2}}.
\]
We sample $W_{ij} (i \leq j)$ from the density $p$ and let $H_{ij} = W_{ij}/\sqrt{N}$. We also sample the spike $\bsx = (x_1, x_2, \dots, x_N)$ so that $\sqrt{N} x_i$'s are i.i.d. Rademacher. The data matrix $M=H+\sqrt{\lambda} \bsx \bsx^T$. From direct calculation, it is straightforward to see that the constants in Theorem \ref{thm:main} are
\[
	F = F_d = \frac{\pi^2}{8}, \quad G = \frac{\pi^4}{4}.
\]
Assume that $\SNR < 1/F = 8/\pi^2$. By Corollary \ref{cor:main}, the limiting error of the LR test is
\beq \label{eq:error_uni}
	\mathrm{erfc} \left( \frac{1}{4} \sqrt{-\log \big( 1- \frac{\pi^2 \SNR}{8} \big) + \frac{\pi^2 \SNR}{8} + \frac{7\pi^4 \SNR^2}{128}} \right).
\eeq

We perform a numerical experiment with $500$ samples of the $32 \times 32$ matrix $M$ under $\bsH_0 (\lambda = 0)$ and $\bsH_1 (\lambda = \SNR)$, respectively, varying the SNR $\SNR$ from $0$ to $0.5$. Since finding the exact log-LR in \eqref{eq:non-Gaussian_LR} for each sample is not viable computationally (as it requires to compute the average of the likelihood ratios with $2^{32}$ different spikes), we apply a Monte Carlo method as follows: for each $M$, we compute
\beq
	\caL^{(\ell)} := \prod_{i \leq j} \frac{p(\sqrt{N} M_{ij}-\sqrt{\SNR N} x^{(\ell)}_i x^{(\ell)}_j)}{p(\sqrt{N} M_{ij})} = \frac{\sech \big( (\pi\sqrt{N} M_{ij}- \pi\sqrt{\SNR N} x^{(\ell)}_i x^{(\ell)}_j )/2\big)}{\sech(\pi\sqrt{N} M_{ij}/2)},
\eeq
for $1 \leq \ell \leq 10^5$, where $\sqrt{N} x^{(\ell)}_i$'s are i.i.d. Rademacher, independent of $\bsx$. We use the test statistic $\ol\caL$, which is the average of $\caL^{(\ell)}$'s, i.e.,
\beq \label{eq:MC_LR}
	\ol\caL := \frac{1}{10^5} \sum_{\ell=1}^{10^5} \caL^{(\ell)}.
\eeq
The behavior of the test statistic $\ol\caL$ under $\bsH_0$ and $\bsH_1$ for the SNR $\SNR=0.3$ and $\SNR=0.4$ can be found in Figure \ref{fig:histogram}. 

\begin{figure}[t]
\centering
    \includegraphics[width=0.9\linewidth]{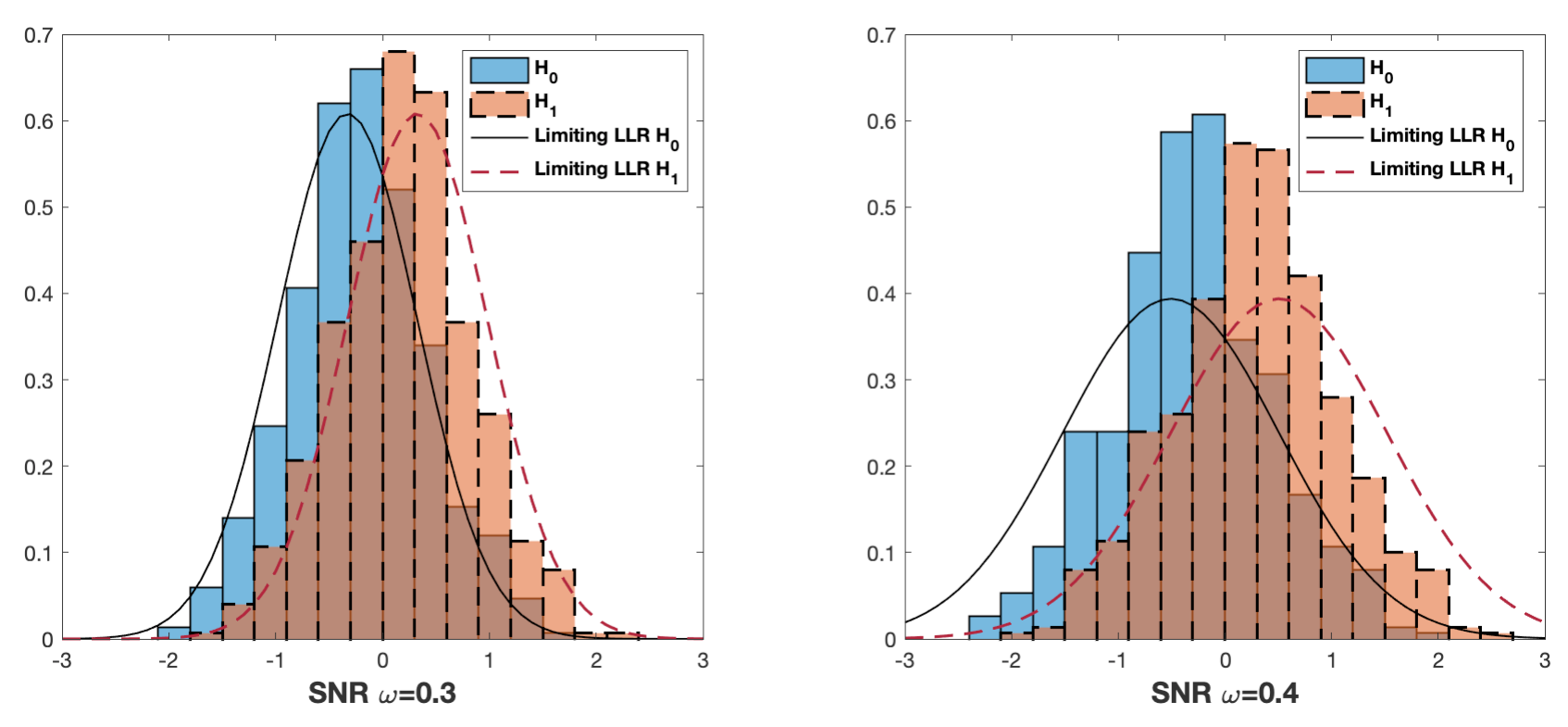}
    \caption{The histogram of the log likelihood ratios under $H_0$ and $H_1$, respectively, for the SNR $\SNR=0.3$ and $\SNR=0.4$ after 500 runs. The limiting log likelihood ratios are also plotted. }
  \label{fig:histogram}
\end{figure}

To compare the error from the numerical experiment with the theoretical error, we consider the test in which we accept $\bsH_0$ if $\ol\caL \leq 1$ and reject $\bsH_1$ otherwise. The result can be found in Figure \ref{fig:main_inf_real}.(a), which shows that the empirical LLR error closely matches the theoretical limiting error in \eqref{eq:error_uni}.

In the algorithm proposed in \cite{chung2019weak}, the data matrix was pre-transformed by
\[
	\wt M_{ij} = \sqrt{\frac{2}{N}} \tanh \left( \frac{\pi \sqrt{N}}{2} M_{ij} \right).
\]
The test statistic based on the LSS of $\wt M$ was used for the hypothesis test, which is given by
\[ \begin{split}
	\wt L_{\SNR} = -\log \det \left( (1+\frac{\pi^2 \SNR}{8})I - \sqrt{\frac{\pi^2 \SNR}{8}} \tM \right) + \frac{\pi^2 \SNR N}{16} 
	+ \frac{\pi \sqrt{\SNR}}{2 \sqrt{2}} \Tr \tM + \frac{\pi^2 \SNR}{16} (\Tr \tM^2 - N).
\end{split} \]
We accept $\bsH_0$ if 
\[
	\wt L_{\SNR} \leq -\log \left(1-\frac{\pi^2 \SNR}{8} \right) - \frac{3 \pi^4 \SNR^2}{512}
\]
and reject $\bsH_0$ otherwise. The limiting error of this test is
\beq \label{eq:error_LSS}
	\mathrm{erfc} \left( \frac{1}{4} \sqrt{-\log \big( 1- \frac{\pi^2 \SNR}{8} \big) + \frac{\pi^2 \SNR}{8}} \right).
\eeq
Note that the error of this algorithm is larger than that of the LR test in \eqref{eq:error_uni}, as $\mathrm{erfc}(z)$ is a decreasing function of $z$. See also Figure \ref{fig:main_inf_real}.(a). We remark that the difference between the empirical error and the limiting error increases as the SNR $\SNR$ increases, and it is possibly due to the finite size effect, i.e., the contribution from the small terms that were ignored in the asymptotic expansions.

\begin{figure}[tb]
\centering
\begin{subfigure}{0.49\textwidth}
\centering
    \includegraphics[width=0.8\linewidth]{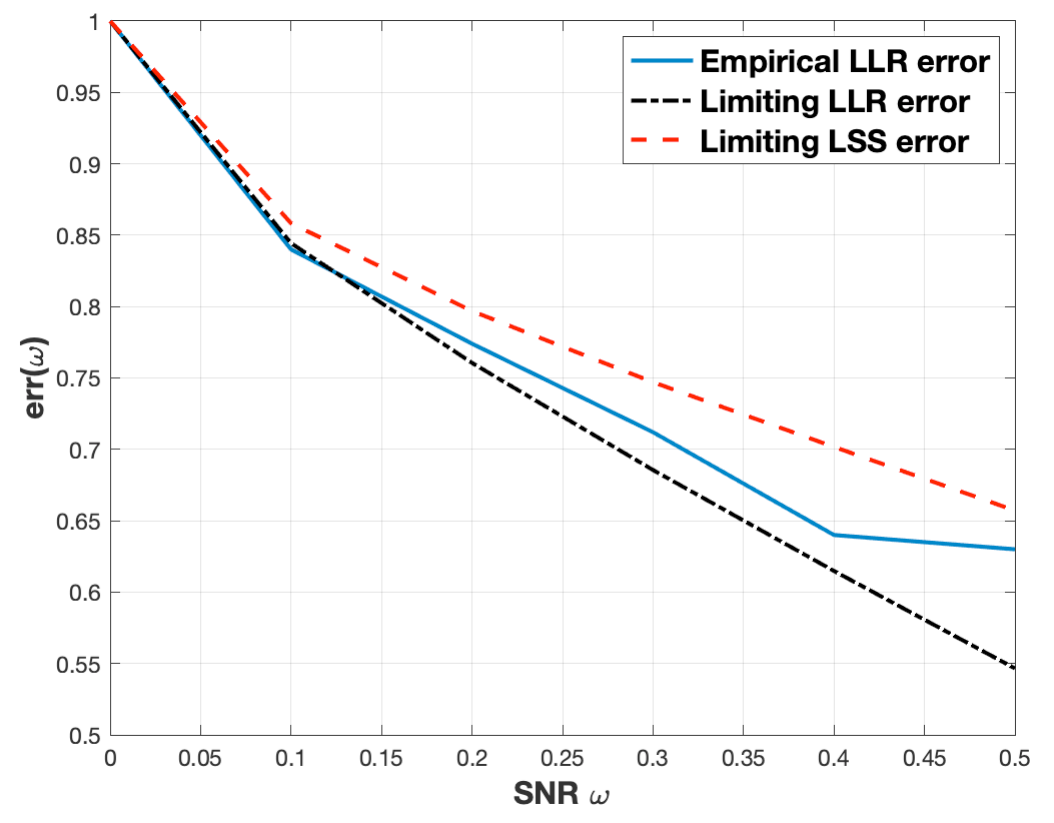}
    \caption{Rademacher prior}
  \label{fig:rademacher}
\end{subfigure}
\begin{subfigure}{0.49\textwidth}
\centering
    \includegraphics[width=0.8\linewidth]{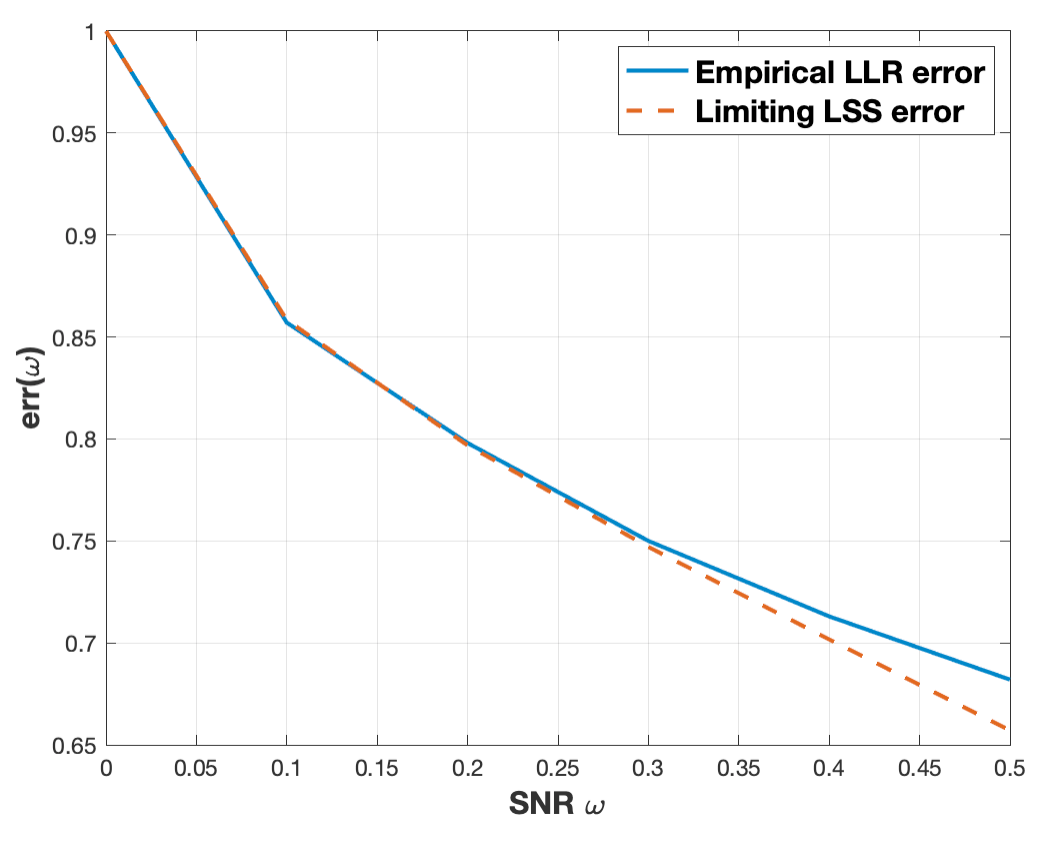}
    \caption{Spherical prior}
  \label{fig:spherical}
\end{subfigure}\\
\caption{The empirical LLR errors (blue solid curves) with the varying SNR $\SNR$. (a) The empirical LLR error with Rademacher prior is compared with its limiting error of LLR (black dash-dot curve) and that of linear spectral statistics (LSS) (red dashed curve). The empirical LLR error closely matches its limiting error and is lower than the limiting LSS error. (b) The empirical LLR error with spherical prior is compared with the limiting error of LSS. The empirical LLR error with spherical prior closely matches the limiting LSS error. }
\label{fig:main_inf_real}
\end{figure}

To numerically test the conjecture introduced in the last paragraph of Section \ref{subsec:main}, we perform a simulation with the spherical prior instead of the Rademacher prior. More precisely, we use the same density for the noise matrix, while we sample the spike by first sampling $\wt x_i$ to be i.i.d. standard Gaussians and then normalizing it by
\[
	x_i = \frac{\wt x_i}{\sqrt{\wt x_1^2 + \wt x_2^2 + \dots + \wt x_N^2}},
\]
which lets the spike $\bsx = (x_1, x_2, \dots, x_N)$ be distributed uniformly on the unit sphere. 
In the numerical experiment with the spherical prior, we again generated $500$ samples of the $32 \times 32$ matrix $M$ under $\bsH_0 (\lambda = 0)$ and $\bsH_1 (\lambda = \SNR)$, respectively, varying the SNR $\SNR$ from $0$ to $0.5$. For the LR test, we used the test statistic in \eqref{eq:MC_LR} for which the random sample spike $\bsx^{(\ell)} = (x^{(\ell)}_1, x^{(\ell)}_2, \dots, x^{(\ell)}_N)$ is drawn uniformly on the unit sphere. The numerical error of the LR test can be found in Figure \ref{fig:main_inf_real}.(b), which shows that the empirical LLR error closely matches the theoretical limiting error of the LSS test in \eqref{eq:error_LSS}. From the numerical experiment, we can check that the error from the LR test depends on the prior, and it also suggests that the LSS-based test is optimal in the sense that it minimizes the sum of the Type-I error and the Type-II error.

\section{Proof of main theorem} \label{sec:proof}

In this section, we provide the detailed proof of Theorem \ref{thm:main}. We use the following shorthand notations for the proof.

\begin{nrem}
	For $X$ and $Y$, which can be deterministic numbers and/or random variables depending on $N$, we use the notation $X = \caO(Y)$ if for any (small) $\varepsilon > 0$ and (large) $D > 0$ there exists $N_0 \equiv N_0 (\varepsilon, D)$ such that $\p(|X|>N^{\varepsilon} |Y|) < N^{-D}$ whenever $N > N_0$.
	  For an event $\Omega$, we say that $\Omega$ holds with high probability if for any (large) $D > 0$ there exists $N_0 \equiv N_0 (D)$ such that $\p(\Omega^c) < N^{-D}$ whenever $N > N_0$.
		For a sequence of random variables, the notation $\Rightarrow$ denotes the convergence in distribution as $N\rightarrow\infty.$ 
\end{nrem}

The following lemma will be frequently used in the proof of Theorem \ref{thm:main}.

\begin{lemma} \label{lem:stochastic_bound}
For any $1 \leq i, j \leq N$ and for any positive integer $s$,
\[
	H_{ij} = \caO(N^{-\frac{1}{2}}), \quad \frac{p^{(s)}(\sqrt{N}H_{ij})}{p(\sqrt{N}H_{ij})}, \frac{p_d^{(s)}(\sqrt{N}H_{ii})}{p_d(\sqrt{N}H_{ii})} = \caO(1).
\]
Moreover, if we define
\[
	h(x) := -\frac{p'(x)}{p(x)}, \quad h_d(x) = -\frac{p_d'(x)}{p_d(x)},
\]
then
\[
	h^{(s)}(\sqrt{N}H_{ij}), h_d^{(s)}(\sqrt{N}H_{ii}) = \caO(1).
\]
\end{lemma}

Lemma \ref{lem:stochastic_bound} can be proved by applying standard arguments using Markov's inequality and the finite moment condition in Definition \ref{defn:wigner}. See Appendix \ref{sec:lemma_proof} for the detailed proof.

From Lemma \ref{lem:stochastic_bound}, we find that for any given $\epsilon>0$, $\sqrt{N} H_{ij}$, $p^{(s)}(\sqrt{N}H_{ij})/p(\sqrt{N}H_{ij})$, and $p_d^{(s)}(\sqrt{N}H_{ii})/p_d(\sqrt{N}H_{ii})$ are uniformly bounded by $N^{\epsilon}$ with high probability. Thus, we have the following estimate for the i.i.d. sums of these random variables.

\begin{lemma} \label{lem:LLN}
Define
\beq
	P^{(s)}_{ij} := \frac{p^{(s)}(\sqrt{N} H_{ij})}{p(\sqrt{N} H_{ij})}, \quad (P_d^{(s)})_{ii} := \frac{p_d^{(s)}(\sqrt{N} H_{ii})}{p_d(\sqrt{N} H_{ii})}.
\eeq
For any $s_1, s_2, \dots, s_k \geq 1$,
\beq \begin{split}
	\frac{2}{N(N-1)} \sum_{i<j} P^{(s_1)}_{ij} P^{(s_2)}_{ij} \dots P^{(s_k)}_{ij} &= \E[P^{(s_1)}_{12} P^{(s_2)}_{12} \dots P^{(s_k)}_{12}] + \caO(N^{-1}), \\
	\frac{1}{N} \sum_{i=1}^N (P_d^{(s_1)})_{ii} (P_d^{(s_2)})_{ii} \dots (P_d^{(s_k)})_{ii} &= \E[(P_d^{(s_1)})_{11} (P_d^{(s_2)})_{11} \dots (P_d^{(s_k)})_{11}] + \caO(N^{-\frac{1}{2}}).
\end{split} \eeq
\end{lemma}

See Appendix \ref{sec:lemma_proof} for the detailed proof of Lemma \ref{lem:LLN}.

\subsection{Decomposition of LR} \label{subsec:decomp}

Assume $\bsH_0$. Recall that
\beq \label{eq:LR_formula}
	\caL(M;\SNR) = \frac{1}{2^N} \sum_{\bsx} \prod_{i<j} \frac{p(\sqrt{N} M_{ij}-\sqrt{\SNR N} x_i x_j)}{p(\sqrt{N} M_{ij})} \prod_k \frac{p_d(\sqrt{N} M_{kk}-\sqrt{\SNR N} x_k^2)}{p_d(\sqrt{N} M_{kk})}\,.
\eeq
We first focus on the off-diagonal part of \eqref{eq:LR_formula} and assume $i<j$. Recall the definition of $P^{(s)}_{ij}$ in Lemma \ref{lem:LLN}. Note that $P^{(s)}_{ij}$ is an odd function of $\sqrt{N} M_{ij}$ if $s$ is odd, whereas it is an even function if $s$ is even. Recall that $x_i^2 = N^{-1}$. Since $P^{(s)}_{ij} = \caO(1)$ from Lemma \ref{lem:stochastic_bound}, by the Taylor expansion,
\beq \begin{split}
	\frac{p(\sqrt{N} M_{ij}-\sqrt{\SNR N} x_i x_j)}{p(\sqrt{N} M_{ij})}
	= 1 -\sqrt{\SNR N} P^{(1)}_{ij} x_i x_j + \frac{\SNR P^{(2)}_{ij}}{2N} - \frac{\SNR \sqrt{\SNR}}{6\sqrt{N}} P^{(3)}_{ij} x_i x_j + \frac{\SNR^2 P^{(4)}_{ij}}{24N^2} + \caO(N^{-\frac{5}{2}}),
\end{split} \eeq
uniformly on $i$ and $j$. Taking the logarithm and Taylor expanding it again, we obtain
\[ \begin{split}
	&\log \left( \frac{p(\sqrt{N} M_{ij}-\sqrt{\SNR N} x_i x_j)}{p(\sqrt{N} M_{ij})} \right)\\
	&= -\sqrt{\SNR N} x_i x_j \left( P^{(1)}_{ij} + \frac{\SNR}{6N} \Big( P^{(3)}_{ij} - 3 P^{(1)}_{ij} P^{(2)}_{ij} + 2 (P^{(1)}_{ij})^3 \Big) \right) + \frac{\SNR}{2N} \left( P^{(2)}_{ij} - (P^{(1)}_{ij})^2 \right) \\
	&\qquad + \frac{\SNR^2}{24N^2} \left( P^{(4)}_{ij} -3 (P^{(2)}_{ij})^2 - 4P^{(1)}_{ij} P^{(3)}_{ij} + 12(P^{(1)}_{ij})^2 P^{(2)}_{ij} -6(P^{(1)}_{ij})^4 \right) + \caO(N^{-\frac{5}{2}}).
\end{split} \]

We define $N \times N$ matrices $A$, $B$, $C$ by
\beq \begin{split} \label{eq:ABC}
	& A_{ij}:= -\sqrt{\SNR N} \left( P^{(1)}_{ij} + \frac{\SNR}{6N} \Big( P^{(3)}_{ij} - 3 P^{(1)}_{ij} P^{(2)}_{ij} + 2 (P^{(1)}_{ij})^3 \Big) \right), \quad B_{ij}:= \frac{\SNR}{2N} \left( P^{(2)}_{ij} - (P^{(1)}_{ij})^2 \right), \\
	& C_{ij}:= \frac{\SNR^2}{24N^2} \left( P^{(4)}_{ij} -3 (P^{(2)}_{ij})^2 - 4P^{(1)}_{ij} P^{(3)}_{ij} + 12(P^{(1)}_{ij})^2 P^{(2)}_{ij} -6(P^{(1)}_{ij})^4 \right).
\end{split} \eeq
Note that from Lemma \ref{lem:stochastic_bound}, $A_{ij} = \caO(\sqrt{N})$, $B_{ij} = \caO(N^{-1})$, $C_{ij} = \caO(N^{-2})$.
We then have
\beq \begin{split} \label{eq:decompose}
	&\log \frac{1}{2^N} \sum_{\bsx} \prod_{i<j} \frac{p(\sqrt{N} M_{ij}-\sqrt{\SNR N} x_i x_j)}{p(\sqrt{N} M_{ij})} \\
	&= \log \frac{1}{2^N} \sum_{\bsx} \exp \left( \sum_{i<j} \log \left( \frac{p(\sqrt{N} M_{ij}-\sqrt{\SNR N} x_i x_j)}{p(\sqrt{N} M_{ij})} \right) \right) \\
	&= \log \frac{1}{2^N} \sum_{\bsx} \exp \left( \sum_{i<j} (A_{ij} x_i x_j + B_{ij} + C_{ij}) + \caO(N^{-\frac{1}{2}}) \right) \\
	&= \log \frac{1}{2^N} \sum_{\bsx} \exp \left( \sum_{i<j} A_{ij} x_i x_j + \sum_{i<j} (B_{ij}+C_{ij}) + \caO(N^{-\frac{1}{2}}) \right).
\end{split} \eeq

As briefly explained in Section \ref{subsec:intro_Wigner}, we first analyze the term involving $A_{ij}$, which we call the spin glass term, by applying the strategy of \cite{ALR87} conditional on $B$. We can then verify the part dependent of $B$, which is $\zeta'$ in Proposition \ref{prop:Z_limit}. The sum of $\zeta'$ and the term involving $B$ in \eqref{eq:decompose}, which we call the CLT term, can be written as the sum of i.i.d. random variables that depend only on $B$, and we apply the central limit theorem to prove its convergence. Since $C_{ij} = \caO(N^{-2})$, the sum $\sum_{i<j} C_{ij} = \caO(1)$ and its fluctuation is negligible by the law of large numbers; we call the term involving $C$ in \eqref{eq:decompose} the LLN term.

\subsection{Spin glass term} \label{subsec:spin}

We have the following result for the convergence of the term involving $A$. Let $\E_B$ and $\Var_B$ denote the conditional expectation and the conditional variance given $B$, respectively. 
\begin{prop} \label{prop:Z_limit}
Let $A$ be as defined in \eqref{eq:ABC}. Set 
\beq
	Z:= \frac{1}{2^N} \sum_{\bsx} \exp \left( \sum_{i<j} A_{ij} x_i x_j \right).
\eeq
Then, there exist random variables $\zeta$ and $\zeta'$ such that
\beq \label{eq:decomposition_of_Z}
	\log Z = \log \zeta + \zeta' + \caO(N^{-1}),
\eeq
where $\zeta$ and $\zeta'$ satisfy the following:
\begin{enumerate}
\item The conditional distribution of $\log \zeta$ given $B$ converges (in distribution) to $\caN (-\nu, 2\nu)$, where 
\beq \label{eq:nu}
	\nu:= \sum_{k=3}^\infty \frac{(\SNR F)^k}{4k} = -\frac{1}{4} \left( \log (1- \SNR F) + \SNR F + \frac{\SNR^2 F^2}{2} \right).
\eeq

\item Conditional on $B$,
\beq \label{eq:zeta'_prop}
	\zeta' = \frac{1}{2N^2} \sum_{i<j} \E_B[A_{ij}^2] - \frac{\SNR^2}{24} \E[ (P^{(1)}_{12})^4 ] + U,
\eeq
where $U$ is a random variable whose asymptotic law is a centered Gaussian with variance 
\beq \label{eq:theta}
	\theta:= \frac{\SNR^2}{8} \E \left[ \Var_B \big( (P^{(1)}_{12})^2 \big) \right].
\eeq

\item The fluctuations $\log \zeta$ and $\zeta'$ are asymptotically orthogonal to each other under $L_B^2$, the conditional $L^2$-norm given $B$.

\end{enumerate}
\end{prop}

We will prove Proposition \ref{prop:Z_limit} in Section \ref{sec:indep}.

Here, we would like to recall the definition of the conditional distribution. For random variables $X_1$ and $X_2$ with density, let $f_{X_1, X_2}$ be the joint density function of $X_1$ and $X_2$ and $f_{X_1}$ the marginal density function of $X_1$. Then, the conditional distribution of $X_2$ given $X_1$ has the conditional density function $f_{X_2|X_1}$ satisfying the relation
\[
	f_{X_2|X_1}(x_2|x_1) f_{X_1}(x_1) = f_{X_1, X_2}(x_1, x_2).
\]
Note that $f_{X_2|X_1}(x_2|x_1)$ is the function of $x_1$ and $x_2$.

\begin{rem} \label{rem:conditional}
From the first part of Proposition \ref{prop:Z_limit}, we can also find that $\log \zeta$ converges in distribution to $\caN (-\nu, 2\nu)$ without conditioning on $B$. In general, suppose that the conditional distribution of a sequence of random variables $(X_n)$ given $Y$ converges and the limiting distribution is independent of $Y$. Then, $\E[e^{\ii t X_n}|Y] \to \phi(t)$as $n \to \infty$ for some $\phi(t)$ that is independent of $Y$. Since $\E[ e^{\ii t X_n}] = \E[ \E[e^{\ii t X_n}|Y] ]$, we find that $\E[ e^{\ii t X_n}] \to \phi(t)$ and thus $(X_n)$ converges in distribution to the same limiting distribution.
\end{rem}

\begin{rem} \label{rem:center}
Let $\text{sgn}(t)$ be the sign function, defined as $\text{sgn}(t) = 1$ if $t\geq 0$ and $\text{sgn}(t) = -1$ if $t<0$. Then, $\text{sgn}(H_{ij})$ is a random variable such that $\P(\text{sgn}(H_{ij})=1)=\P(\text{sgn}(H_{ij})=-1)= 1/2$. Since $B_{ij}$ is an even function of $\sqrt{N} H_{ij}$ under $\bsH_0$, we can see that $B_{ij}$ and $\text{sgn} (H_{ij})$ are independent. Since $A_{ij}$ is an odd function of $\sqrt{N} H_{ij}$ under $\bsH_0$, this implies that the conditional density function of $A_{ij}$ given $B_{ij}$ is symmetric (about $0$), and in particular, $\E[A_{ij}^s|B_{ij}] = 0$ for any odd $s \in \mathbb{Z}^+$.
\end{rem}

\begin{example}
To better understand the decomposition in Proposition \ref{prop:Z_limit}, we consider the following example that was considered in Section \ref{sec:ex}. Recall that in this example the density of the normalized entries of the noise matrix is
\[
	p(x) = p_d(x) = \frac{\sech(\pi x/2)}{2} = \frac{1}{e^{\pi x /2} + e^{-\pi x/2}}
\]
and it is assumed that $\SNR < 1/F =8/\pi^2$. The derivatives of $p$ are given by
\[ \begin{split}
	p'(x) &=-\frac\pi{2}\tanh(\pi x/2)\cdot p(x), \quad p''(x)=\frac{\pi^2}{4} \left( 1 - 2\sech^2(\pi x/2) \right) p(x), \\
	p'''(x) &=-\frac{\pi^3}{8}\tanh(\pi x/2) \left( 6\sech^2(\pi x/2) -1\right)\cdot p(x).
\end{split} \]
Then, from the definition in \eqref{eq:ABC},
\[
	A_{ij}=-\sqrt{\omega N}\left(-\frac\pi{2}\tanh\Big(\frac{\pi\sqrt{N} M_{ij}}{2}\Big)-\frac{\omega}{6N}\cdot\frac{\pi^3}{4}\tanh\Big(\frac{\pi\sqrt{N} M_{ij}}{2}\Big) \left( 5\sech^2\Big(\frac{\pi\sqrt{N} M_{ij}}{2}\Big) - 1 \right) \right)
\]
and
\[
	B_{ij}=-\frac{\omega}{2N}\cdot\frac{\pi^2}{4}\sech^2\Big(\frac{\pi\sqrt{N} M_{ij}}{2}\Big).
\]
Applying the identity $1 - \tanh^2 x = \sech^2 x$, we thus find that
\[
	A_{ij}=\frac{\sqrt{N}}{2} \sqrt{\SNR\pi^2 + 8NB_{ij}} \left( 1- \frac{10B_{ij}}{3} - \frac{\SNR\pi^2}{12N} \right) \cdot\text{sgn}(M_{ij}).
\]
(See Remark \ref{rem:center} for a definition of $\text{sgn}(M_{ij})$.) Recall that $Z$ depends only on $A_{ij}$. In Section \ref{sec:indep}, under $\bsH_0$, we will see that $\zeta$ depends only on $\tanh (A_{ij}/N)$ whereas $\zeta'$ depends only on $A_{ij}^2$. Due to the independence between $B_{ij}$ and $\text{sgn} (M_{ij})$ discussed in Remark \ref{rem:center}, we find that the conditional distribution of $\log \zeta$ depends only on the random variables $\text{sgn}(M_{ij})$. Thus, Equation \eqref{eq:decomposition_of_Z} in this case means the decomposition of $Z$, conditional on $B$, into one part that depends only on $\text{sgn} (M_{ij})$ and the other part that depends only on $B_{ij}$. It is then obvious that the fluctuations $\log \zeta$ and $\zeta'$ are asymptotically orthogonal to each other under $L_B^2$, due to their independence.
\end{example}

\subsection{CLT part} \label{subsec:CLT}

From Proposition \ref{prop:Z_limit}, we find that the part involving $B$ in the spin glass term is asymptotically $(2N^2)^{-1} \sum_{i<j} \E_B[A_{ij}^2]$. Thus, combining it with the term involving $B$ in \eqref{eq:decompose}, we are led to consider
\beq \label{eq:B_sum}
	\sum_{i<j} B_{ij} + \frac{1}{2N^2} \sum_{i<j} \E[A_{ij}^2|B_{ij}],
\eeq
which is the sum of i.i.d. random variables that depend only on $B_{ij}$. (The contribution from $U$ in \eqref{eq:zeta'_prop} will be considered separately at the end of the proof.) By the central limit theorem, \eqref{eq:B_sum} converges to a normal random variable whose mean and variance we compute in this subsection.

We first compute the mean of \eqref{eq:B_sum}. By definition,
\beq \label{eq:B_ij}
	\E[B_{ij}] = \frac{\SNR}{2N} \E \left[ P^{(2)}_{ij} - (P^{(1)}_{ij})^2 \right].
\eeq
Note that for any $s \geq 1$,
\beq \label{eq:int_part_1}
	\E\left[ P^{(s)}_{ij} \right] = \int_{\R} \frac{p^{(s)}(t)}{p(t)} \, p(t) \, \dd t = \int_{\R}p^{(s)}(t) \, \dd t = 0,
\eeq
and
\[
	F = \int_{\R} \frac{p'(t)^2}{p(t)} \dd t = \E\left[ (P^{(1)}_{ij})^2 \right].
\]
We also have from the definition of $A_{ij}$ and the law of total expectation, $\E \left[ \E[A_{ij}^2|B_{ij}] \right] = \E[A_{ij}^2]$, that
\beq \begin{split} \label{eq:A^2_sum}
	\E \left[ \E[A_{ij}^2|B_{ij}] \right] = \E[A_{ij}^2] 
	= \SNR N \E\left[ (P^{(1)}_{ij})^2 \right] + \frac{\SNR^2}{3} \E\left[ P^{(1)}_{ij} P^{(3)}_{ij} - 3 (P^{(1)}_{ij})^2 P^{(2)}_{ij} + 2 (P^{(1)}_{ij})^4 \right] + O(N^{-1}).
\end{split} \eeq
Using integration by parts, we obtain that
\beq \begin{split} \label{eq:int_part_2}
	\E\left[ (P^{(1)}_{ij})^4 \right] &= \int_{\R} \frac{p'(t)^4}{p(t)^3} \dd t = -\frac{1}{2} \int_{\R} \left( \frac{1}{p(t)^2} \right)' (p'(t)^3) \dd t = \frac{3}{2} \int_{\R} \frac{p'(t)^2 p''(t)}{p(t)^2} \dd t = \frac{3}{2} \E\left[ (P^{(1)}_{ij})^2 P^{(2)}_{ij} \right].
\end{split} \eeq
Note that the boundary term in the integration by parts used in \eqref{eq:int_part_2} vanishes since $p'(t^{\pm}_n)^3/p(t^{\pm}_n)^2 \to 0$ for some sequences $(t^+_n)_{n=1, 2, \dots}$ and $(t^-_n)_{n=1, 2, \dots}$ such that $t^+_n \to \infty$ and $t^-_n \to -\infty$, respectively, which can be checked from the fact that
\[
	\int_{\R} \left| \frac{p'(t)^3}{p(t)^2} \right| \leq C_3 \int_{\R} |t|^{r_3} p(t) \dd t = C_3 \E[ (\sqrt N H_{ij})^{r_3} ] < \infty.
\]
(See Assumption \ref{assump:wigner} for the definition of the constants $C_3$ and $r_3$.)
Thus, from \eqref{eq:A^2_sum},
\beq \label{eq:A^2_sum_2}
	\E \left[ \E[A_{ij}^2|B_{ij}] \right] = \SNR N \E\left[ (P^{(1)}_{ij})^2 \right] + \frac{\SNR^2}{3} \E\left[ P^{(1)}_{ij} P^{(3)}_{ij} \right] + O(N^{-1}),
\eeq
and we conclude from \eqref{eq:B_sum}, \eqref{eq:B_ij}, \eqref{eq:int_part_1}, and \eqref{eq:A^2_sum} that
\beq \begin{split} \label{eq:B_mean}
	\E \left[ \sum_{i<j} B_{ij} + \frac{1}{2N^2} \sum_{i<j} \E[A_{ij}^2|B_{ij}] \right] &= \frac{\SNR^2}{6N^2} \sum_{i<j} \E\left[ P^{(1)}_{ij} P^{(3)}_{ij} \right] + O(N^{-1}) \\
	&= \frac{\SNR^2}{12} \E\left[ P^{(1)}_{12} P^{(3)}_{12} \right] + O(N^{-1}).
\end{split} \eeq

We next compute the variance of \eqref{eq:B_sum}. Since $\E[B_{ij} A_{ij}^2 | B_{ij}] = B_{ij} \E[A_{ij}^2 | B_{ij}]$ by the definition of the conditional expectation, we have
\beq \begin{split} \label{eq:B^2}
	\E\left[ \left( B_{ij} + \frac{\E[A_{ij}^2|B_{ij}]}{2N^2} \right)^2 \right] = \E[ B_{ij}^2] + \frac{1}{N^2} \E[A_{ij}^2 B_{ij}] + \frac{1}{4N^4} \E[ (\E[A_{ij}^2|B_{ij}])^2].
\end{split} \eeq
Again, by definition,
\beq \label{eq:39}
	\E[B_{ij}^2] = \frac{\SNR^2}{4N^2} \E \left[ (P^{(2)}_{ij})^2 -2 (P^{(1)}_{ij})^2 P^{(2)}_{ij} + (P^{(1)}_{ij})^4 \right],
\eeq
\beq \label{eq:40}
	\frac{1}{N^2} \E[A_{ij}^2 B_{ij}] = \frac{\SNR^2}{2N^2} \E \left[ (P^{(1)}_{ij})^2 \Big( P^{(2)}_{ij} - (P^{(1)}_{ij})^2 \Big) \right] + O(N^{-3}),
\eeq
and
\beq
	\frac{1}{4N^4} \E[ (\E[A_{ij}^2|B_{ij}])^2] = \frac{\SNR^2}{4N^2} \E[ (\E[(P^{(1)}_{ij})^2|B_{ij}])^2] + O(N^{-3}).
\eeq
Adding \eqref{eq:39} and \eqref{eq:40}, we find that
\beq \label{eq:B^2_first_two}
	\E[ B_{ij}^2] + \frac{1}{N^2} \E[A_{ij}^2 B_{ij}] = \frac{\SNR^2}{4N^2} \E \left[ (P^{(2)}_{ij})^2 - (P^{(1)}_{ij})^4 \right] + O(N^{-3}).
\eeq
Using the identity
\beq \label{eq:var_identity}
	\E[(P^{(1)}_{ij})^4] = \E \left[ (\E[(P^{(1)}_{ij})^2|B_{ij}])^2 + \Var \big( (P^{(1)}_{ij})^2|B_{ij} \big) \right],
\eeq
and combining \eqref{eq:B^2} and \eqref{eq:B^2_first_two}, we have
\beq \label{eq:B_var}
	\E\left[ \left( B_{ij} + \frac{\E[A_{ij}^2|B_{ij}]}{2N^2} \right)^2 \right] = \frac{\SNR^2}{4N^2} \E \left[ (P^{(2)}_{ij})^2 - \Var \big( (P^{(1)}_{ij})^2|B_{ij} \big) \right] + O(N^{-3}).
\eeq

From \eqref{eq:B_ij} and \eqref{eq:A^2_sum}, we know that
\[
	\E \left[ B_{ij} + \frac{\E[A_{ij}^2|B_{ij}]}{2N^2} \right] = O(N^{-2}),
\]
and hence
\[
	\sum_{i<j} \E\left[ \left( B_{ij} + \frac{\E[A_{ij}^2|B_{ij}]}{2N^2} \right)^2 \right]  - \Var \left( \sum_{i<j} B_{ij} + \frac{1}{2N^2} \sum_{i<j} \E[A_{ij}^2|B_{ij}] \right) = O(N^{-2}).
\]
We thus conclude from \eqref{eq:B_mean} and \eqref{eq:B_var} that
\beq \begin{split} \label{eq:limiting_normal}
	\sum_{i<j} B_{ij} + \frac{1}{2N^2} \sum_{i<j} \E[A_{ij}^2|B_{ij}]  \Rightarrow \caN \left( \frac{\SNR^2}{12} \E\left[ P^{(1)}_{12} P^{(3)}_{12} \right], \frac{\SNR^2}{8} \E \left[ (P^{(2)}_{12})^2 - \Var \big( (P^{(1)}_{12})^2|B_{12} \big) \right] \right).
\end{split} \eeq

\subsection{LLN part} \label{subsec:LLN}

From the definition of $C_{ij}$ in \eqref{eq:ABC}, applying Lemma \ref{lem:LLN}, we find that
\beq
	\sum_{i<j} C_{ij} = \frac{\SNR^2}{48} \E\left[ P^{(4)}_{12} -3 (P^{(2)}_{12})^2 - 4P^{(1)}_{12} P^{(3)}_{12} + 12(P^{(1)}_{12})^2 P^{(2)}_{12} -6(P^{(1)}_{12})^4 \right] + \caO(N^{-1}).
\eeq
Since $\E[P^{(4)}_{12}]=0$ from \eqref{eq:int_part_1}, applying the integration by parts formulas \eqref{eq:int_part_2} and
\beq \begin{split} \label{eq:int_part_3}
	\E\left[ (P^{(1)}_{ij})^2 P^{(2)}_{ij} \right] &= \int_{\R} \frac{p'(t)^2 p''(t)}{p(t)^2} \dd t = -\int_{\R} \left( \frac{1}{p(t)} \right)' (p'(t) p''(t)) \dd t \\
	&= \int_{\R} \frac{p''(t)^2 + p'(t) p'''(t)}{p(t)} \dd t = \E\left[ (P^{(2)}_{ij})^2 + P^{(1)}_{ij} P^{(3)}_{ij} \right],
\end{split} \eeq
we find that
\beq \label{eq:LLN}
	\sum_{i<j} C_{ij} = -\frac{\SNR^2}{48} \E[P^{(1)}_{12} P^{(3)}_{12}] + \caO(N^{-1}).
\eeq

\subsection{Diagonal part} \label{subsec:diagonal}

For the diagonal part in \eqref{eq:LR_formula}, it is not hard to see that
\beq \begin{split} 
	&\log \prod_k \frac{p_d(\sqrt{N} M_{kk}-\sqrt{\SNR N} x_k^2)}{p_d(\sqrt{N} M_{kk})} \\
	&= \sum_k \left( - \sqrt{\frac{\SNR}{N}} \frac{p_d'(\sqrt{N} M_{kk})}{p_d(\sqrt{N} M_{kk})} + \frac{\SNR}{2N} \left( \frac{p_d''(\sqrt{N} M_{kk}) p_d(\sqrt{N} M_{kk}) - p_d'(\sqrt{N} M_{kk})^2}{p_d(\sqrt{N} M_{kk})^2} \right) \right) 
	+ \caO(N^{-\frac{1}{2}}),
\end{split} \eeq
which does not depend on $\bsx$. Applying the integration by parts formulas, one can evaluate the mean and the variance to show that
\beq \label{eq:diagonal_Taylor}
	\log \prod_k \frac{p_d(\sqrt{N} M_{kk}-\sqrt{\SNR N} x_k^2)}{p_d(\sqrt{N} M_{kk})} \Rightarrow \caN (-\frac{\SNR F_d}{2}, \SNR F_d).
\eeq

\subsection{Proof of Theorem \ref{thm:main}}

We now combine Proposition \ref{prop:Z_limit}, Equations \eqref{eq:decompose}, \eqref{eq:limiting_normal}, \eqref{eq:LLN}, and \eqref{eq:diagonal_Taylor} to find that $\log \caL(M;\lambda)$ converges to the Gaussian whose mean is
\beq \begin{split}
	&\frac{1}{4} \left( \log (1-\SNR F) + \SNR F+ \frac{\SNR^2 F^2}{2}\right) - \frac{\SNR^2}{24} \E[ (P^{(1)}_{12})^4 ] 
	+ \frac{\SNR^2}{12} \E[ P^{(1)}_{12} P^{(3)}_{12}] -\frac{\SNR^2}{48} \E[P^{(1)}_{12} P^{(3)}_{12}] -\frac{\SNR F_d}{2} \\
	&= \frac{1}{4} \left( \log (1-\SNR F) + \SNR F+ \frac{\SNR^2 F^2}{2} \right) - \frac{\SNR^2}{16} \E[ (P^{(2)}_{12})^2] -\frac{\SNR F_d}{2} = -\rho
\end{split} \eeq
where we also applied the integration by parts formulas \eqref{eq:int_part_2} and \eqref{eq:int_part_3}. (See also Remark \ref{rem:conditional}.) Similarly, we can also find that the variance of the limiting Gaussian is 
\beq \begin{split}
	&-\frac{1}{2} \left( \log (1-\SNR F) + \SNR F+ \frac{\SNR^2 F^2}{2}  \right) + \frac{\SNR^2}{8} \E\left[\Var_B \big( (P^{(1)}_{12})^2 \big) \right] \\
	&\qquad \qquad + \frac{\SNR^2}{8} \E \left[ (P^{(2)}_{12})^2 - \Var_B \big( (P^{(1)}_{12})^2 \big) \right] + \SNR F_d = 2\rho.
\end{split} \eeq
This concludes the proof of Theorem \ref{thm:main}.

\section{Asymptotic independence of fluctuations} \label{sec:indep}

In this section, we prove Proposition \ref{prop:Z_limit}. The key idea for the proof is to decompose the fluctuation of $Z$ into two parts. First, notice that
\[
	\exp \left( A_{ij} x_i x_j \right) = \cosh \frac{A_{ij}}{N} + N x_i x_j \sinh \frac{A_{ij}}{N} ,
\]
which can be checked by considering the cases $x_i x_j = N^{-1}$ and $x_i x_j = -N^{-1}$ separately.
Thus, if we let
\beq \label{eq:def_zeta}
	\zeta := \frac{1}{2^N} \sum_{\bsx} \prod_{i<j} \left( 1 + N x_i x_j \tanh \frac{A_{ij}}{N} \right),
\eeq
then it is direct to see that
\beq
	\log Z - \log \zeta = \log \prod_{i<j} \cosh \left( \frac{A_{ij}}{N} \right).
\eeq
Recall that $A_{ij} = \caO(\sqrt{N})$. By Taylor expanding the right hand side, we get
\beq
	\log Z - \log \zeta = \sum_{i<j} \left( \frac{A_{ij}^2}{2N^2} - \frac{A_{ij}^4}{12N^4} \right) + \caO(N^{-1}),
\eeq
and we define
\beq \label{eq:zeta'} 
	\zeta' := \sum_{i<j} \left( \frac{A_{ij}^2}{2N^2} - \frac{A_{ij}^4}{12N^4} \right).
\eeq
It is then obvious that $\log Z = \log \zeta + \zeta' + \caO(N^{-1})$. 

In the rest of this section, we first find the fluctuation of $\zeta'$, which can be readily obtained by the CLT. Then, we adopt the strategy for the proof of Proposition 2.2 in \cite{ALR87} to conclude the proof of Proposition \ref{prop:Z_limit} by finding the asymptotic fluctuation of $\log \zeta$.

\subsection{Conditional law of $\zeta'$} \label{subsec:zeta'}

Recall the definition of $\zeta'$ in \eqref{eq:zeta'}. Since 
\[
	A_{ij} = -\sqrt{\SNR N} P^{(1)}_{ij} + \caO(N^{-1/2})
\]
and $P^{(1)}_{ij} = \caO(1)$, we find that
\[
	\sum_{i<j} \frac{A_{ij}^4}{N^4} \to \frac{\SNR^2}{2} \E[ (P^{(1)}_{12})^4 ].
\]
We want to apply the central limit limit theorem to the array $(A_{ij}^2/(2N^2): 1 \leq i < j \leq N)$, conditional on $B$. Since $P^{(1)}_{ij}$ and $N B_{ij}$ are $N$-independent random variables, we find that $\Var_B ( (P^{(1)}_{ij})^2 )$ is a non-negative, $N$-independent random variable. In particular, $\E[\Var_B ( (P^{(1)}_{ij})^2)]$ is a non-negative, $N$-independent constant.

Suppose that $\E[\Var_B ( (P^{(1)}_{ij})^2)] > 0$. Then, since
\[
	\frac{2}{N(N-1)} \sum_{i<j} \Var_B \big( \frac{A_{ij}^2}{\SNR N} \big) \to \E[\Var_B ( (P^{(1)}_{ij})^2)] > 0,
\]
we find that there exists a constant $c$, independent of $N$, such that
\[
	\sum_{i<j} \Var_B \big( A_{ij}^2 \big) > cN^4
\]
almost surely. Further, since
\[
	\E[ (A_{ij}^2 - \E_B[A_{ij}^2])^4] = O(N^4),
\]
Lyapunov's condition for the central limit theorem is satisfied with the array $(A_{ij}^2: 1 \leq i < j \leq N)$ as
\[
	\left( \sum_{i<j} \E\left[ \Big( A_{ij}^2 - \E_B \big[ A_{ij}^2 \big] \Big)^4 \right] \right) \Big/ \left( \sum_{i<j} \Var_B \big( A_{ij}^2 \big) \right)^2 = O(N^{-2})
\]
almost surely. Since Lyapunov's condition is scaling-invariant, we find that the central limit theorem holds for $(A_{ij}^2/(2N^2): 1 \leq i < j \leq N)$, conditional on $B$, and the asymptotic law of the random variable
\beq
	U:= \zeta' - \left( \frac{1}{2N^2} \sum_{i<j} \E_B[A_{ij}^2] - \frac{\SNR^2}{24} \E[ (P^{(1)}_{12})^4 ] \right)
\eeq
is given by a centered Gaussian with variance $\theta$ defined in \eqref{eq:theta}, since
\[
	\left( \sum_{i<j} \Var_B \big( \frac{A_{ij}^2}{2N^2} \big) \right) \to \frac{\SNR^2}{8} \E[ \Var_B \big( (P^{(1)}_{12})^2 \big)].
\]

If $\E[\Var_B ( (P^{(1)}_{ij})^2)] = 0$, then $(P^{(1)}_{ij})^2 = \E_B [ (P^{(1)}_{ij})^2 ]$ almost surely, and hence $A_{ij}^2 = \E_B [A_{ij}^2] + \caO(1)$. We then find that
\[
	\left( \sum_{i<j} \Var_B \big( \frac{A_{ij}^2}{2N^2} \big) \right) \to 0
\]
almost surely, which implies that
\[
	\frac{1}{2N^2} \sum_{i<j} A_{ij}^2 - \frac{1}{2N^2} \sum_{i<j} \E_B[A_{ij}^2] \to 0
\]
almost surely. Thus, in this case, $U$ converges to $0$, and hence the second part of Proposition \ref{prop:Z_limit} holds with $\theta = 0$. This completes the proof of the second part of Proposition \ref{prop:Z_limit}.

\subsection{Conditional law of $\zeta$} \label{subsec:zeta}

In this subsection, we prove the first part of Proposition \ref{prop:Z_limit} by a combinatorial approach. Consider the product $\prod_{i<j} \left( 1 + N x_i x_j \tanh (A_{ij}/N) \right)$ in \eqref{eq:def_zeta} as a polynomial of $x_k's$. A term in this polynomial is of the form
\beq \label{eq:zeta_poly}
	N^{\ell} x_{i_1} x_{j_1} x_{i_2} x_{j_2} \dots x_{i_{\ell}} x_{j_{\ell}} \tanh \frac{A_{i_1 j_1}}{N} \tanh \frac{A_{i_2 j_2}}{N} \dots \tanh \frac{A_{i_{\ell} j_{\ell}}}{N}.
\eeq
If the multiplicity of any of the $x_{i_1}, \dots, x_{i_{\ell}}, x_{j_1}, \dots, x_{j_{\ell}}$ in \eqref{eq:zeta_poly} is odd, then the term \eqref{eq:zeta_poly} is negligible in \eqref{eq:def_zeta}; for instance, if the multiplicity of the $x_{i_1}$ is odd, then the contribution of the term \eqref{eq:zeta_poly} from the case $x_{i_1} = 1/\sqrt{N}$ and that from the case $x_{i_1} = -1/\sqrt{N}$ exactly cancel each other. It in particular implies that \eqref{eq:zeta_poly} is actually of the form
\beq \begin{split} \label{eq:zeta_poly2}
	&N^{\ell} x_{i_1} x_{j_1} x_{i_2} x_{j_2} \dots x_{i_{\ell}} x_{j_{\ell}} \tanh \frac{A_{i_1 j_1}}{N} \tanh \frac{A_{i_2 j_2}}{N} \dots \tanh \frac{A_{i_{\ell} j_{\ell}}}{N} \\
	&= \tanh \frac{A_{i_1 j_1}}{N} \tanh \frac{A_{i_2 j_2}}{N} \dots \tanh \frac{A_{i_{\ell} j_{\ell}}}{N},
\end{split} \eeq
where the pair $(i_k, j_k)$ does not repeat. We then find the one-to-one correspondence between the right hand side of \eqref{eq:zeta_poly2} and a simple and closed graph on the vertex set $[n] := \{ 1, 2, \dots, n \}$, i.e., with no self-edges, repeated edges, or vertices of odd degree. Define for a simple and closed graph $\Gamma$ 
\[
	w(\Gamma):= \prod_{(i, j) \in E(\Gamma)} \tanh \Big( \frac{A_{ij}}{N} \Big),
\]
where $E(\Gamma)$ denotes the edge set of $\Gamma$. A straightforward calculation shows that 
\beq\label{eq:zeta_by_graphs}
	\zeta = \sum_{\Gamma} w(\Gamma). 
\eeq

We further decompose $w(\Gamma)$ by considering simple loops, or cycles, contained in $\Gamma$. (A simple loop is a graph where every vertex has degree $2$.) For example, if the edges of $\Gamma$ are $(1, 2)$, $(1, 3)$, $(2, 3)$, $(3, 4)$, $(3, 5)$, $(4, 5)$, then the degree of the vertex $3$ is $4$ and thus $\Gamma$ is not a simple loop. However, $\Gamma$ can be decomposed into simple loops $\gamma_1$ and $\gamma_2$ whose sets of edges are $\{ (1, 2), (1, 3), (2, 3) \}$ and $\{ (3, 4), (3, 5), (4, 5) \}$, respectively.

Let us denote a simple loop by the lowercase $\gamma$. Note that for a simple loop $\gamma$, $\E_B[w(\gamma)] = 0$, since the $A_{ij}$'s are independent and their density functions given $B$ is symmetric. (See Remark \ref{rem:center}.) We introduce the following result from \cite{ALR87}:
\begin{lemma}\label{lem:zeta_close}
	The random variable $\zeta - \prod_{\gamma} (1+w(\gamma))$ converges in probability to 0.
\end{lemma}
For the proof of Lemma \ref{lem:zeta_close}, see the end of Section 3 in \cite{ALR87}.

For the proof of the first part of Proposition \ref{prop:Z_limit}, we claim for $\prod_{\gamma} (1+w(\gamma))$ that
\beq \label{eq:1+w_expand}
	\prod_{\gamma} \left( 1+ w(\gamma) \right) = \exp \left( \sum_{\gamma} \Big( w(\gamma) - \frac{w(\gamma)^2}{2} \Big)(1+ \caO(N^{-1})) \right).
\eeq
To prove the claim, we first notice that if $|E(\gamma)|=k$, i.e., a simple loop $\gamma$ has $k$ edges, then
\beq\label{eq:max_bound}
    \E |w(\gamma)|^2 \leq \frac{(\SNR F + \epsilon)^k}{N^k}, \quad \E |w(\gamma)|^s = O(N^{-\frac{sk}{2}}) 
\eeq
for any $\epsilon > 0$ and any positive integer $s$, which can be checked from an inequality $|\tanh x| \leq |x|$. Similarly, applying Markov's inequality, we can show that
\beq\label{eq:weight_bound}
	\max_{\gamma:|E(\gamma)| = k} |w(\gamma)|^s = \caO(N^{-\frac{s}{2}|E(\gamma)|}), \quad \max_{\gamma:|E(\gamma)| = k} \E_B |w(\gamma)|^s = \caO(N^{-\frac{s}{2}|E(\gamma)|})
\eeq
for any positive integer $s$, where $E(\gamma)$ is the set of edges of $\gamma$. We thus find from the Taylor expansion that
\[
	1+ w(\gamma) = \exp \left( \Big( w(\gamma) - \frac{w(\gamma)^2}{2} \Big)(1+ \caO(N^{-1})) \right).
\]

From the expansion \eqref{eq:1+w_expand}, we find that it is required to prove convergence results for the sums of $w(\gamma)$ and $w(\gamma)^2$. We claim the following:
\begin{lemma}\label{lem:w_gamma}
Let $\xi:=\sum_{\gamma} w(\gamma)$, $\eta:=\sum_{\gamma} w(\gamma)^2$. Let $\nu$ be given as in \eqref{eq:nu}. Then, as $N \to \infty$,
\begin{enumerate}
    \item[(i)] The conditional law $\xi$ given $B$ converges in probability to a centered Gaussian with variance $2\nu$, and
    \item[(ii)] $\eta$ converges in probability to $2\nu$.
\end{enumerate} 
\end{lemma}

See Appendix \ref{sec:proof_Lemma_w_gamma} for the proof of Lemma \ref{lem:w_gamma}.

Since
\[
	\prod_{\gamma} \left( 1+ w(\gamma) \right) = \exp \left( \xi - \frac{\eta}{2} + \caO(N^{-1}) \right),
\]
it is immediate to prove the first part of Proposition \ref{prop:Z_limit} from Lemmas \ref{lem:zeta_close} and \ref{lem:w_gamma}. Furthermore, since the source of the fluctuation of $\log \zeta$ is $\xi = \sum_{\gamma} w(\gamma)$, which is orthogonal to $A_{ij}^2$ that generates the fluctuation of $\zeta'$ under $L_B^2$, we find that the third part of Proposition \ref{prop:Z_limit}, the asymptotic orthogonality of $\log \zeta$ and $\zeta'$, holds.

\section{Conclusion and future works} \label{sec:conclusion}

In this paper, we studied the detection problem for the spiked Wigner model. Assuming the Rademacher prior, we proved the asymptotic normality of the log likelihood ratio with precise formulas for the mean and the variance of the limiting Gaussian. We computed the limiting error of the likelihood ratio test, and we conducted numerical experiments to compare the results from the experiments with the theoretical error. We also obtained the corresponding results for the spiked IID model where the noise is asymmetric.

A natural extension of the current work is to prove the corresponding results for other models, especially the spiked Wigner models with the spherical prior or the sparse Rademacher prior, and prove the conjecture that the limiting distribution of the log likelihood ratio depends on the prior. It is also of interest to extend the techniques in this paper to spiked rectangular models.

\subsubsection*{Acknowledgments}
We thank the anonymous referees for their helpful comments and suggestions. The third author thanks to Ji Hyung Jung and Seong-Mi Seo for helpful comments. The work of the second author was performed when he was affiliated with Stochastic Analysis and Application Research Center (SAARC) at KAIST. The first author was supported in part by National Research Foundation of Korea (NRF) grant funded by the Korea government(MSIT) (No. 2021R1C1C11008539 and 2023R1A2C1005843).
The second author and the third author were supported in part by National Research Foundation of Korea (NRF) grant funded by the Korea government(MSIT) (No. 2019R1A5A1028324).

\begin{appendix}

\section{Proof of Lemmas \ref{lem:stochastic_bound} and \ref{lem:LLN}} \label{sec:lemma_proof}

In this section, we provide the detail of the proof of the high probability estimates, Lemmas \ref{lem:stochastic_bound} and \ref{lem:LLN}.

\begin{proof}[Proof of Lemma \ref{lem:stochastic_bound}]
We first prove the estimate for $H_{ij}$. For given $\epsilon, D > 0$, choose an integer $D' > D/\epsilon$. By applying Markov's inequality,
\[
	\p( |H_{ij}| > N^{-\frac{1}{2} + \epsilon}) = \p( N^{\frac{D'}{2}} |H_{ij}|^{D'} > N^{D' \epsilon}) \leq N^{-D'\epsilon} N^{\frac{D'}{2}} \E[|H_{ij}|^{D'}] < C_{D'} N^{-D'\epsilon} < N^{-D}
\]
for any sufficiently large $N$, where we used the finite moment assumption for $H_{ij}$ in Definition \ref{defn:wigner}. This proves that $H_{ij} = \caO(N^{-\frac{1}{2}})$. Note that it also implies that $\sqrt{N} H_{ij} = \caO(1)$.

From the polynomial bound assumption for $p^{(s)}/p$ in Assumption \ref{assump:wigner}, 
\[
	\left| \frac{p^{(s)}(\sqrt{N}H_{ij})}{p(\sqrt{N}H_{ij})} \right| \leq C_s |\sqrt{N}H_{ij}|^{r_s}.
\]
Thus, for any $\epsilon', D > 0$
\[
	\p \left( \left| \frac{p^{(s)}(\sqrt{N}H_{ij})}{p(\sqrt{N}H_{ij})} \right| > N^{\epsilon'} \right) \leq \p \left( |\sqrt{N}H_{ij}|> \left( \frac{N^{\epsilon'}}{C_s} \right)^{\frac{1}{r_s}} \right)
\]
for any sufficiently large $N$. Since $\epsilon'$ is arbitrary, this proves that $\frac{p^{(s)}(\sqrt{N}H_{ij})}{p(\sqrt{N}H_{ij})} = \caO(1)$. The estimate for $\frac{p_d^{(s)}(\sqrt{N}H_{ij})}{p_d(\sqrt{N}H_{ij})}$ can be proved in a similar manner.

The last part of the lemma follows from that $h^{(s)}$ and $h_d^{(s)}$ can be written as a polynomial of $(p^{(s)}/p)$'s and $(p_d^{(s)}/p_d)$'s, respectively.
\end{proof}

\begin{proof}[Proof of Lemma \ref{lem:LLN}]
We prove the first inequality for the case $k=1$ only; the general case can be handled in a similar manner. 

Set $s := s_1$. From Lemma \ref{lem:stochastic_bound}, we find that for any $\epsilon>0$, $P_{ij}^{(s)}$ are uniformly bounded by $N^{\epsilon}$ with high probability. Define
\[
	\Omega_{ij} := \{ |P_{ij}^{(s)}| < N^{\epsilon} \}, \quad \Omega := \bigcap_{i, j} \Omega_{ij}.
\]
Since $\Omega_{ij}$ holds with high probability, by taking a union bound, $\Omega$ also holds with high probability. For any $D>0$, applying Hoeffding's inequality (or high order Markov's inequality) on $\Omega$, we obtain
\beq \label{eq:LLN_proof_step1}
	\p \left( \left| \sum_{i<j} \left( P_{ij}^{(s)} - \E[P_{ij}^{(s)}|\Omega_{ij}] \right) \right| \geq N^{1+2\epsilon} \right) < N^{-D}
\eeq
for any sufficiently large $N$. Moreover, since $P_{ij}^{(s)}$ is polynomially bounded, for any $D>0$,
\[
	\left| \E[ P_{ij}^{(s)} \mathbf{1}_{\Omega_{ij}^c}] \right| \leq \E[ |P_{ij}^{(s)}|^2] \E[ \mathbf{1}_{\Omega_{ij}^c}] \leq (C_s)^2 \E[ |\sqrt{N} H_{ij}|^{r_s D'}] \p(\Omega_{ij}^c) < N^{-D}
\]
for any sufficiently large $N$, which shows that
\beq \label{eq:LLN_proof_step2}
	\E[P_{ij}^{(s)}] - \E[P_{ij}^{(s)}|\Omega_{ij}] = \E[P_{ij}^{(s)}] - \frac{\E[ P_{ij}^{(s)} \mathbf{1}_{\Omega_{ij}}]}{\p(\Omega_{ij})} = \E[ P_{ij}^{(s)} \mathbf{1}_{\Omega_{ij}^c}] + O(N^{-D}) = O(N^{-D}).
\eeq
Combining \eqref{eq:LLN_proof_step1} and \eqref{eq:LLN_proof_step2},
\[
	\sum_{i<j} P_{ij}^{(s)} = \E[P_{12}^{(s)}] + \caO(N),
\]
and after multiplying both sides by $\frac{2}{N(N-1)}$, we conclude that the first inequality of Lemma \ref{lem:LLN} holds.
\end{proof}

\section{Proof of Lemma \ref{lem:w_gamma}} \label{sec:proof_Lemma_w_gamma}

We first prove the limit of $\E[\eta]$.
Let 
\[
	\xi_k:= \sum_{\gamma:|E(\gamma)|=k} w(\gamma).
\]
Note that for distinct simple loops $\gamma_1$ and $\gamma_2$, $w(\gamma_1)$ and $w(\gamma_2)$ are orthogonal with respect to both $L^2_B$ and $L^2$. To check this, assume that the edge $(i_1, j_1)$ is in $\gamma_1$ but not in $\gamma_2$. Then, from the independence of $A_{ij}$'s,
\[ \begin{split}
	\E[w(\gamma_1) w(\gamma_2)] &= \E \left[ \tanh \frac{A_{i_1 j_1}}{N} \prod_{(i, j): (i, j) \in E(\gamma_1) \cup E(\gamma_2), (i, j) \neq (i_1, j_2)} \tanh \frac{A_{ij}}{N} \right] \\
	&= \E \left[ \tanh \frac{A_{i_1 j_1}}{N} \right] \E \left[ \prod_{(i, j): (i, j) \in E(\gamma_1) \cup E(\gamma_2), (i, j) \neq (i_1, j_2)} \tanh \frac{A_{ij}}{N} \right] = 0.
\end{split} \]
A similar idea can also be used to prove the orthogonality of $w(\gamma_1)$ and $w(\gamma_2)$ with respect to $L^2_B$.

We now consider $\E ||\xi_k ||_{L^2_B}$. Let $P(N,k) :=\frac{N!}{(N-k)!}$. Since the number of possible simple loops of length $k$ in the vertex set $[N]$ is $P(N,k)/(2k)$, we have from the orthogonality of $\gamma$'s that
\beq \begin{split} \label{eq:second_moment} 
	\E ||\xi_k ||_{L^2_B}^2 &= \sum_{\gamma:|E(\gamma)|=k} \prod_{(i, j) \in E(\gamma) } \E \left[ \tanh^2 \frac{A_{ij}}{N} \right] = \frac{P(N,k)}{2k} \E \left[ \tanh^2 \frac{A_{12}}{N} \right]^k \\
	&\to \frac{N^k}{2k}\E \left[ (\SNR N) \frac{(P^{(1)}_{12})^2}{N^2} \right]^k = \frac{(\SNR F)^k}{2k}.
\end{split} \eeq
Also, for any given $\epsilon>0$,
\beq\label{eq:domination}
    \E ||\xi_k ||_{L^2_B}^2 \leq \frac{(\SNR F + \epsilon)^k}{2k}.
\eeq
Since $\xi = \sum_{k \leq N} \xi_k$ and $\SNR F <1$, by the orthogonality and the dominated convergence theorem we have
\beq\label{eq:limiting_second_moment}
\E[\eta] = \E [\xi^2] = \E [\E_B [\xi^2]] \to \sum_{k=3}^\infty \frac{(\SNR F)^k}{2k} = 2\nu
\eeq
as $N \to \infty$. 

We next consider $\Var (\E_B [\xi^2])$. Our goal is to show that $\Var (\E_B [\xi^2]) \to 0$ as $N \to \infty$, and for this purpose, it suffices to prove that $\Var (\E_B [(\xi^{(j)})^2]) \to 0$ for any fixed $j$ where we let $\xi^{(j)} := \sum_{k \leq j} \xi_k$, since \eqref{eq:second_moment} and \eqref{eq:domination} imply that
\[
	\sum_{k \geq j} \xi_k \to 0
\]
as $j \to \infty$ in $L^2$, uniformly in $N$.
Note that
\beq\label{eq:variance_L^2-norm}
	\Var (\E_B [(\xi^{(j)})^2]) = \sum_{\gamma_1, \gamma_2: |E(\gamma_1)|, |E(\gamma_2)| \leq j} \Cov \left( \E_B[w(\gamma_1)^2], \E_B[w(\gamma_2)^2] \right).
\eeq
The covariance vanishes if there is no common edge, and otherwise
\beq\label{eq:bound_cov}
\Cov \Big( \E_B [w(\gamma_1)^2], \E_B [w(\gamma_2)^2] \Big) \leq J^m (\SNR F)^{|E(\gamma_1)|+|E(\gamma_2)|-2m} N^{-(|E(\gamma_1)|+|E(\gamma_2)|)}
\eeq
for large $N$, where $J$ is a fixed number greater than $N^{-2} \Var (A_{ij}^2)$ and $m$ is the number of common edges of $\gamma_1$ and $\gamma_2$. However, the number of isomorphism types of $\gamma_1$ and $\gamma_2$ is finite for a fixed $j$, and the number of possible choices of $\gamma_1$ and $\gamma_2$ for each isomorphism type is $O(N^{|E(\gamma_1)|+|E(\gamma_2)|-2}$), verifying that the right hand side of \eqref{eq:variance_L^2-norm} converges to $0$. Similarly, we can show that $\Var (\eta) \to 0$ as $N \to \infty$, which combined with \eqref{eq:limiting_second_moment} proves (ii).

It remains to estimate $\E_B [\xi^{2s}]$. Our goal is to show that
\beq \label{eq:Gaussian_moment}
	\E_B[\xi^{2s}] - \frac{(2s)!}{2^s s!} (\E_B[\xi^2])^s \to 0
\eeq
in probability. The coefficient $((2s)!)/(2^s s!)$ is the number of pairings of $[2s]$, or equivalently the $2s$-th moment of the Gaussian. Since the conditional odd moments of $\xi$ given $B$ are zero, (i) immediately follows from \eqref{eq:limiting_second_moment} and \eqref{eq:Gaussian_moment}. For the proof of \eqref{eq:Gaussian_moment}, we notice that it suffices to prove the statement for $\xi^{(j)}$ as in the paragraph above.

By conditional independence of $w(\gamma)$'s, we have
\beq \begin{split} \label{eq:higher_moment}
\E_B[ (\xi^{(j)})^{2s}] - \frac{(2s)!}{2^s s!} (\E_B[(\xi^{(j)})^2])^s 
= \sum_{\gamma_i: i\leq 2s} \E_B \left[\prod_{i \leq 2s} w(\gamma_i) \right] - \frac{(2s)!}{2^s s!} \sum_{\gamma_i: i\leq 2s} \prod_{i \leq s} \E_B [w(\gamma_{2i-1})w(\gamma_{2i})] ,
\end{split} \eeq
where $\gamma_i$'s run over simple loops with length at most $j$. In the first term, however, the contribution from the case where $\gamma_i$'s are paired into non-intersecting double loops cancel with the second term. Thus, we have
\beq\label{eq:canceled}
\E_B [(\xi^{(j)})^{2s}] - \frac{(2s)!}{2^s s!} (\E_B[(\xi^{(j)})^2])^s = \sum_M \E_B [w(M)],
\eeq
where $M$ runs over multigraphs with even edge multiplicity, which can be described as a union of $2s$ loops of length at most $j$, but not as a disjoint union of $s$ double loops. From an estimate
\[
	\max_{\gamma:|E(M)| = k} \E_B |w(M)|^s = \caO(N^{-\frac{s}{2}|E(M)|}),
\]
which can be proved as \eqref{eq:weight_bound}, each summand in the right hand side of \eqref{eq:canceled} is $\caO(N^{-s})$. However, since the sum of the order of all the vertices, which is equal to $4s$, is at least $4|V(M)|+4$, by the assumptions on $M$. Therefore, the contribution of each isomorphism type to the sum is $\caO(N^{-s+|V(M)|})$, hence $\caO(N^{-1})$. Since there are finitely many isomorphism types of $M$, we obtain \eqref{eq:higher_moment}. This completes the proof of of Lemma \ref{lem:w_gamma}.

\section{Proof of Theorem \ref{thm:iid}} \label{sec:iid_proof}

In this section, we provide the detail of the proof of Theorem \ref{thm:iid}.

\subsection{Preliminaries} \label{subsec:prelim_iid}

Recall that we have the following estimates from Lemma \ref{lem:stochastic_bound}: For any $1 \leq i, j \leq N$ and for any $s \geq 1$,
\beq
	X_{ij} = \caO(N^{-\frac{1}{2}}), \quad h^{(s)}(\sqrt{N}X_{ij}) = \caO(1).
\eeq
In particular, for any $\epsilon>0$, $\sqrt{N} X_{ij}$ and $h^{(s)}(\sqrt{N}X_{ij})$ are uniformly bounded by $N^{\epsilon}$ with high probability. 
We set
\beq
	Q^{(s)}_{ij} := \frac{p^{(s)}(\sqrt{N} Y_{ij})}{p(\sqrt{N} Y_{ij})},
\eeq
where we denote by $p^{(s)}$ the $s$-th derivative of $p$. We also have from Lemma \ref{lem:LLN} that 
\beq \begin{split}
	\frac{1}{N(N-1)} \sum_{i\neq j} Q_{ij}^{(s_1)} Q_{ij}^{(s_2)} \dots Q_{ij}^{(s_k)} &= \E[Q_{12}^{(s_1)} Q_{12}^{(s_2)} \dots Q_{12}^{(s_k)}] + \caO(N^{-1}), \\
	\frac{1}{N} \sum_{i=1}^N Q_{ii}^{(s_1)} Q_{ii}^{(s_2)} \dots Q_{ii}^{(s_k)} &= \E[Q_{11}^{(s_1)} Q_{11}^{(s_2)} \dots Q_{11}^{(s_k)}] + \caO(N^{-\frac{1}{2}}),
\end{split} \eeq
for any $s_1, s_2, \dots, s_k \geq 1$.

Assume $\bsH_0$. The likelihood ratio is
\beq \begin{split} \label{eq:LR_decomp_iid}
	\caL(Y;\SNR) &= \frac{1}{2^N} \sum_{\bsx} \prod_{i, j=1}^N \frac{p(\sqrt{N} Y_{ij}-\sqrt{\SNR N} x_i x_j)}{p(\sqrt{N} Y_{ij})} \\
	&= \frac{1}{2^N} \sum_{\bsx} \prod_{i \neq j} \frac{p(\sqrt{N} Y_{ij}-\sqrt{\SNR N} x_i x_j)}{p(\sqrt{N} Y_{ij})} \prod_k \frac{p(\sqrt{N} Y_{kk}-\sqrt{\SNR N} x_k^2)}{p(\sqrt{N} Y_{kk})}\,.
\end{split} \eeq
As in the spiked Wigner case, the diagonal part can be handled separately, since it does not depend on the spike or the off-diagonal part.
Since $Q^{(s)}_{ij}$ is an odd function of $\sqrt{N} Y_{ij}$ if $s$ is odd and an even function if $s$ is even, by the Taylor expansion,
\beq \begin{split}
	&\frac{p(\sqrt{N} Y_{ij}-\sqrt{\SNR N} x_i x_j)}{p(\sqrt{N} Y_{ij})} \\
	&= 1 -\sqrt{\SNR N} Q^{(1)}_{ij} x_i x_j + \frac{\SNR Q^{(2)}_{ij}}{2N} - \frac{\SNR \sqrt{\SNR}}{6\sqrt{N}} Q^{(3)}_{ij} x_i x_j + \frac{\SNR^2 Q^{(4)}_{ij}}{24N^2} + \caO(N^{-\frac{5}{2}}),
\end{split} \eeq
uniformly on $i$ and $j$. Following the decomposition in \eqref{eq:decompose}, we thus get
\beq \begin{split} \label{eq:iid_decomp}
	\log \frac{1}{2^N} \sum_{\bsx} \prod_{i\neq j} \frac{p(\sqrt{N} Y_{ij}-\sqrt{\SNR N} x_i x_j)}{p(\sqrt{N} Y_{ij})} 
	= \log \frac{1}{2^N} \sum_{\bsx} \exp \left( \sum_{i < j} A^*_{ij} x_i x_j + \sum_{i < j} (B^*_{ij}+C^*_{ij}) + \caO(N^{-\frac{1}{2}}) \right),
\end{split} \eeq
where we let
\beq \begin{split}
	& A^*_{ij}:= -\sqrt{\SNR N} \bigg( Q^{(1)}_{ij} + Q^{(1)}_{ji}  + \frac{\SNR}{6N} \Big( Q^{(3)}_{ij} + Q^{(3)}_{ji} - 3 Q^{(1)}_{ij} Q^{(2)}_{ij} - 3 Q^{(1)}_{ji} Q^{(2)}_{ji} + 2 (Q^{(1)}_{ij})^3 + 2 (Q^{(1)}_{ji})^3 \Big) \bigg), \\
	& B^*_{ij}:= \frac{\SNR}{2N} \left( Q^{(2)}_{ij} + Q^{(2)}_{ji} - (Q^{(1)}_{ij})^2 - (Q^{(1)}_{ji})^2 \right), \\
	& C^*_{ij}:= \frac{\SNR^2}{24N^2} \bigg( Q^{(4)}_{ij} + Q^{(4)}_{ji} -3 (P^{(2)}_{ij})^2 -3 (P^{(2)}_{ji})^2 - 4Q^{(1)}_{ij} Q^{(3)}_{ij} - 4Q^{(1)}_{ji} Q^{(3)}_{ji} \\
	&\qquad \qquad \qquad \qquad + 12(Q^{(1)}_{ij})^2 Q^{(2)}_{ij} + 12(Q^{(1)}_{ji})^2 Q^{(2)}_{ji} -6(P^{(1)}_{ij})^4 -6(P^{(1)}_{ji})^4 \bigg).
\end{split} \eeq
Note that from Lemma \ref{lem:stochastic_bound},
\beq
	A^*_{ij} = \caO(\sqrt{N}), \quad B^*_{ij} = \caO(N^{-1}), \quad C^*_{ij} = \caO(N^{-2}).
\eeq

\subsection{Spin glass part} \label{subsec:spin_iid}

We have the following result for the convergence of the term involving $A^*$. Let $\E_{B^*}$ and $\Var_{B^*}$ denote the conditional expectation and the conditional variance given $B^*$, respectively. The counterpart of Proposition \ref{prop:Z_limit} is as follows:
\begin{prop} \label{prop:Z_lim_iid}
Set
\beq
	Z^*:= \frac{1}{2^N} \sum_{\bsx} \exp \left( \sum_{i<j} A^*_{ij} x_i x_j \right).
\eeq
Then, there exist random variables $\zeta^*$ and $(\zeta^*)'$ such that
\beq
	\log Z^* = \log \zeta^* + (\zeta^*)' + \caO(N^{-1}),
\eeq
where $\zeta^*$ and $(\zeta^*)'$ satisfy the following
\begin{enumerate}
\item The conditional distribution of $\log \zeta^*$ given $B^*$ converges in distribution to $\caN (-\nu^*, 2\nu^*)$, with 
\beq \label{eq:nu*}
	\nu^*:= \sum_{k=3}^\infty \frac{(2\SNR F)^k}{4k} = -\frac{1}{4} \left( \log (1- 2\SNR F) + 2\SNR F + 2 \SNR^2 F^2 \right).
\eeq

\item Conditional on $B^*$,
\beq
	(\zeta^*)' = \frac{1}{2N^2} \sum_{i<j} \E_{B^*} [(A^*_{ij})^2] - \frac{\SNR^2}{12} \E[ (Q^{(1)}_{12})^4 ] - \frac{\SNR^2}{4} \E[ (Q^{(1)}_{12})^2 ]^2 + U^*,
\eeq
where $U^*$ is a random variable whose asymptotic law is a centered Gaussian with variance 
\beq
	\theta^* := \frac{\SNR^2}{8} \E \left[ \Var_{B^*} \big( (Q^{(1)}_{12} + Q^{(1)}_{12})^2 \big) \right].
\eeq

\item The fluctuations $\log \zeta^*$ and $(\zeta^*)'$ are asymptotically orthogonal to each other under $L_{B^*}^2$, the conditional $L^2$-norm given $B^*$.

\end{enumerate}
\end{prop}

Proposition \ref{prop:Z_lim_iid} is proved in Section \ref{sec:indep_iid}.

\subsection{CLT part} \label{subsec:CLT_iid}

From Proposition \ref{prop:Z_lim_iid}, we find that the terms involving $B^*$ in the right hand side of \eqref{eq:iid_decomp} are
\beq \label{eq:B*_sum}
	\sum_{i<j} B^*_{ij} + \frac{1}{2N^2} \sum_{i<j} \E[(A^*_{ij})^2|B^*_{ij}].
\eeq
which is the sum of i.i.d. random variables that depend only on $B^*_{ij}$. By the central limit theorem, it converges to a normal random variable whose mean and variance we compute in this subsection.

We first compute the mean of \eqref{eq:B*_sum}. By definition,
\[
	\E[B^*_{ij}] = \frac{\SNR}{N} \E \left[ Q^{(2)}_{ij} - (Q^{(1)}_{ij})^2 \right] = -\frac{\SNR F}{N}
\]
and using integration by parts formulas in \eqref{eq:int_part_1} and \eqref{eq:int_part_2}, we find that
\beq \begin{split}
	&\E \left[ \E[(A^*_{ij})^2|B^*_{ij}] \right] = \E[(A^*_{ij})^2] 
	= 2\SNR N \E\left[ (Q^{(1)}_{ij})^2 \right] + \frac{2\SNR^2}{3} \E\left[ Q^{(1)}_{ij} Q^{(3)}_{ij} - 3 (Q^{(1)}_{ij})^2 Q^{(2)}_{ij} + 2 (Q^{(1)}_{ij})^4 \right] + O(N^{-1}) \\
	&= 2\SNR N \E\left[ (Q^{(1)}_{ij})^2 \right] + \frac{2\SNR^2}{3} \E\left[ Q^{(1)}_{ij} Q^{(3)}_{ij} \right] + O(N^{-1}),
\end{split} \eeq
where we used the independence of $Q^{(s)}_{ij}$ and $Q^{(s)}_{ji}$ and also that $\E[Q^{(1)}_{ij}] = 0$. We thus get
\beq \begin{split} \label{eq:B*_mean}
	\E \left[ \sum_{i<j} B^*_{ij} + \frac{1}{2N^2} \sum_{i<j} \E[(A^*_{ij})^2|B^*_{ij}] \right] = \frac{\SNR^2}{6} \E\left[ Q^{(1)}_{12} Q^{(3)}_{12} \right] + O(N^{-1}).
\end{split} \eeq

We next compute the the variance of \eqref{eq:B*_sum}. We have from the definition that
\[
	\E[(B^*_{ij})^2] = \frac{\SNR^2}{2N^2} \E \left[ (Q^{(2)}_{ij})^2 -2 (Q^{(1)}_{ij})^2 Q^{(2)}_{ij} + (Q^{(1)}_{ij})^4 \right] + \frac{\SNR^2 F^2}{2N^2},
\]
\[
	\frac{1}{N^2} \E[(A^*_{ij})^2 B^*_{ij}] = \frac{\SNR^2}{N^2} \E \left[ (Q^{(1)}_{ij})^2 \Big( Q^{(2)}_{ij} - (Q^{(1)}_{ij})^2 \Big) \right] - \frac{\SNR^2 F^2}{N^2} + O(N^{-3}),
\]
and
\[
	\frac{1}{4N^4} \E[ (\E[(A^*_{ij})^2|B^*_{ij}])^2] = \frac{\SNR^2}{4N^2} \E[ (\E[(Q^{(1)}_{ij} + (Q^{(1)}_{ji})^2|B^*_{ij}])^2] + O(N^{-3}).
\]
We apply the identity
\[
	\E[(Q^{(1)}_{ij} + Q^{(1)}_{ji})^4] = \E \left[ (\E[(Q^{(1)}_{ij} + Q^{(1)}_{ji})^2|B^*_{ij}])^2 + \Var \big( (Q^{(1)}_{ij} + Q^{(1)}_{ji})^2|B^*_{ij} \big) \right],
\]
which corresponds to \eqref{eq:var_identity}, and also
\[
	\E[(Q^{(1)}_{ij} + Q^{(1)}_{ji})^4] = 2 \E[(Q^{(1)}_{ij})^4] + 6 \E[(Q^{(1)}_{ij})^2]^2\,.
\]
We then obtain
\beq \begin{split} \label{eq:B*_var}
	\E\left[ \left( B^*_{ij} + \frac{\E[(A^*_{ij})^2|B^*_{ij}]}{2N^2} \right)^2 \right] = \frac{\SNR^2}{4N^2} \E \left[ 2(Q^{(2)}_{ij})^2 - \Var \big( (Q^{(1)}_{ij})^2|B^*_{ij} \big) \right] + \frac{\SNR^2 F^2}{N^2} + O(N^{-3}).
\end{split} \eeq

We thus conclude from \eqref{eq:B*_mean} and \eqref{eq:B*_var} that
\beq \begin{split} \label{eq:B*_normal}
	&\sum_{i<j} B^*_{ij} + \frac{1}{2N^2} \sum_{i<j} \E[(A^*_{ij})^2|B^*_{ij}] \\
	&\qquad \Rightarrow \caN \left( \frac{\SNR^2}{6} \E\left[ Q^{(1)}_{12} Q^{(3)}_{12} \right], \frac{\SNR^2}{8} \E\left[2(Q^{(2)}_{12})^2 - \Var_{B^*} \big( (Q^{(1)}_{12} + Q^{(1)}_{21})^2 \big) \right]\right).
\end{split} \eeq

\subsection{LLN part and the diagonal part} \label{subsec:LLN_iid}

From the definition of $C^*_{ij}$ in \eqref{eq:ABC} and the integration by parts formula in \eqref{eq:int_part_3}, we find that
\beq \label{eq:LLN_iid}
	\sum_{i<j} C^*_{ij} = -\frac{\SNR^2}{24} \E[Q^{(1)}_{12} Q^{(3)}_{12}] + \caO(N^{-1}).
\eeq

For the diagonal part in \eqref{eq:LR_decomp_iid}, it can be readily checked that
\beq \label{eq:diag_iid}
	\log \prod_k \frac{p(\sqrt{N} Y_{kk}-\sqrt{\SNR N} x_k^2)}{p(\sqrt{N} Y_{kk})} \Rightarrow \caN (-\frac{\SNR F}{2}, \SNR F).
\eeq

\subsection{Proof of Theorem \ref{thm:iid}}

We now combine Proposition \ref{prop:Z_lim_iid}, Equations \eqref{eq:B*_normal}, \eqref{eq:LLN_iid}, and \eqref{eq:diag_iid} to find that $\log \caL(Y;\lambda)$ in \eqref{eq:LR_decomp_iid} converges to the Gaussian whose mean is
\beq \begin{split}
	&\frac{1}{4} \left( \log (1-2\SNR F) + 2\SNR F+ 2\SNR^2 F^2 \right) - \frac{\SNR^2}{12} \E[ (Q^{(1)}_{12})^4 ] - \frac{\SNR^2 F^2}{4} + \frac{\SNR^2}{6} \E[ Q^{(1)}_{12} Q^{(3)}_{12}] -\frac{\SNR^2}{24} \E[Q^{(1)}_{12} Q^{(3)}_{12}] -\frac{\SNR F}{2} \\
	&= \frac{1}{4} \left( \log (1-\SNR F) + \SNR^2 F^2 \right) - \frac{\SNR^2}{8} \E[ (Q^{(2)}_{12})^2] = -\rho^*
\end{split} \eeq
and variance
\beq \begin{split}
	&-\frac{1}{2} \left( \log (1-2\SNR F) + 2\SNR F+ 2\SNR^2 F^2 \right) + \frac{\SNR^2}{8} \E\left[\Var_{B^*} \big( (Q^{(1)}_{12} + Q^{(1)}_{21})^2 \big) \right] \\
	&\qquad \qquad + \frac{\SNR^2}{8} \E \left[ 2(Q^{(2)}_{12})^2 - \Var_{B^*} \big( (Q^{(1)}_{12} + Q^{(1)}_{21})^2 \big) \right] + \frac{\SNR^2 F^2}{2} + \SNR F = 2\rho^*.
\end{split} \eeq
This proves Theorem \ref{thm:iid}.

\section{Proof of Proposition \ref{prop:Z_lim_iid}} \label{sec:indep_iid}

In this section, we prove Proposition \ref{prop:Z_lim_iid} by following the strategy in Section \ref{sec:indep}. We decompose the fluctuation of $Z^*$ into two parts. Let
\beq
	\zeta^* := \frac{1}{2^N} \sum_{\bsx} \prod_{i<j} \left( 1 + N x_i x_j \tanh \frac{A^*_{ij}}{N} \right).
\eeq
From the direct calculation involving the Taylor expansion as in Section \ref{sec:indep}, we get
\beq
	\log Z^* - \log \zeta^* = \sum_{i<j} \left( \frac{(A^*_{ij})^2}{2N^2} - \frac{(A^*_{ij})^4}{12N^4} \right) + \caO(N^{-1}),
\eeq
and we define
\beq \label{eq:zeta*'} 
	(\zeta^*)' := \sum_{i<j} \left( \frac{(A^*_{ij})^2}{2N^2} - \frac{(A^*_{ij})^4}{12N^4} \right).
\eeq
By definition, $\log Z^* = \zeta^* + (\zeta^*)' + \caO(N^{-1})$. 

For the second part of Proposition \ref{prop:Z_lim_iid}, we notice that the array $(\frac{(A^*_{ij})^2}{2N^2}: 1 \leq i < j \leq N)$ almost surely satisfies Lyapunov's condition for the central limit theorem conditional on $B^*$. Thus, if we define the random variable
\beq
	U^* = (\zeta^*)' -\left( \frac{1}{2N^2} \sum_{i<j} \E_{B^*} [(A^*_{ij})^2] - \frac{\SNR^2}{12} \E[ (Q^{(1)}_{12})^4 ] - \frac{\SNR^2}{4} \E[ (Q^{(1)}_{12})^2 ]^2 \right),
\eeq
its asymptotic law is a centered Gaussian with variance 
\beq
	\theta^* := \frac{\SNR^2}{8} \E \left[ \Var_{B^*} \big( (Q^{(1)}_{12} + Q^{(1)}_{12})^2 \big) \right].
\eeq 
This proves the second part of Proposition \ref{prop:Z_lim_iid}.

Recall that $\Gamma$ is a simple graph on the vertex set $[n] := \{ 1, 2, \dots, n \}$ with no self-edges and no vertices of odd degree and $E(\Gamma)$ is the edge set of $\Gamma$. We define 
\[
	w^*(\Gamma):= \prod_{(i, j) \in E(\Gamma)} \tanh \Big( \frac{A^*_{ij}}{N} \Big),
\]
which gives us
\[
	\zeta^* = \sum_{\Gamma} w^*(\Gamma). 
\]

We now introduce the results on $w^*$ that correspond to Lemmas \ref{lem:zeta_close} and \ref{lem:w_gamma}. Recall that we denote simple loops by the lowercase $\gamma$.
\begin{lemma} \label{lem:zeta*}
Let $\xi^*:=\sum_{\gamma} w^*(\gamma)$, $\eta^*:=\sum_{\gamma} w^*(\gamma)^2$. Let $\nu$ be given as in \eqref{eq:nu*}. Then, as $N \to \infty$,
\begin{enumerate}
	\item[(i)] The random variable $\zeta^* - \prod_{\gamma} (1+w^*(\gamma))$ converges in probability to 0,
  \item[(ii)] The conditional law $\xi^*$ given $B$ converges in probability to a centered Gaussian with variance $2\nu^*$, and
  \item[(iii)] $\eta^*$ converges in probability to $2\nu$.
\end{enumerate} 
\end{lemma}

As in the proof of Proposition \ref{prop:Z_limit}, from the relation
\[
	\prod_{\gamma} \left( 1+ w^*(\gamma) \right) = \exp \left( \sum_{\gamma} \Big( w^*(\gamma) - \frac{w^*(\gamma)^2}{2} \Big)(1+ \caO(N^{-1})) \right) = \exp \left( \xi^* - \frac{\eta^*}{2} + \caO(N^{-1}) \right),
\]
the first part of Proposition \ref{prop:Z_lim_iid} follows from Lemma \ref{lem:zeta*}. The third part of Proposition \ref{prop:Z_lim_iid} is also obvious since $(A^*_{ij})^2$, which generates fluctuation of $(\zeta^*)'$, is orthogonal to $\tanh \frac{A^*_{ij}}{N}$ under $L_{B^*}^2$.

We now prove Lemma \ref{lem:zeta*}. %The first part is from the result in Section 3 of \cite{ALR87}. 
In order to prove the limit of $\E[\eta^*]$, we let
\[
	\xi^*_k:= \sum_{\gamma:|E(\gamma)|=k} w^*(\gamma).
\]
Recall that $P(N,k) =\frac{N!}{(N-k)!}$ and the number of possible simple loops of length $k$ in the vertex set $[N]$ is $\frac{P(N,k)}{2k}$. We then have
\beq \begin{split}
	\E ||\xi^*_k ||_{L^2_{B^*}}^2 &= \sum_{\gamma:|E(\gamma)|=k} \prod_{(i, j) \in E(\gamma) } \E \left[ \tanh^2 \frac{A^*_{ij}}{N} \right] = \frac{P(N,k)}{2k} \E \left[ \tanh^2 \frac{A^*_{ij}}{N} \right]^k \\
	&\to \frac{N^k}{2k}\E \left[ (\SNR N)^k \frac{(Q^{(1)}_{ij} + Q^{(1)}_{ji})^2}{N^2} \right]^k = \frac{(2\SNR F)^k}{2k}.
\end{split} \eeq
Moreover, for any given $\epsilon>0$, there exists a deterministic $N'_0 \in \N$ such that for all $N \geq N_0$,
\beq
    \E ||\xi^*_k ||_{L^2_{B^*}}^2 \leq \frac{(2\SNR F + \epsilon)^k}{2k}.
\eeq
Thus, if $2 \SNR F <1$, by the orthogonality and the dominated convergence theorem we have
\beq
\E[\eta^*] = \E [(\xi^*)^2] = \E [\E_{B^*} [(\xi^*)^2]] \to \sum_{k=3}^\infty \frac{(2\SNR F)^k}{2k} = 2\nu^*
\eeq
as $N \to \infty$. Now, the proof of the second and the third parts of Lemma \ref{lem:zeta*} is the verbatim copy of the proof of Lemma \ref{lem:w_gamma} in Appendix \ref{sec:proof_Lemma_w_gamma} and we omit the detail.

\section{Sketch of the proof for the Spiked IID matrices} \label{sec:asymmetric}

In this section, we provide the sketch of the proof for Theorem \ref{thm:iid}. Recall that
\beq \begin{split} \label{eq:iid_LR}
	\caL(Y;\SNR) &= \frac{1}{2^N} \sum_{\bsx} \prod_{i, j=1}^N \frac{p(\sqrt{N} Y_{ij}-\sqrt{\SNR N} x_i x_j)}{p(\sqrt{N} Y_{ij})} \\
	&= \frac{1}{2^N} \sum_{\bsx} \prod_{i \neq j} \frac{p(\sqrt{N} Y_{ij}-\sqrt{\SNR N} x_i x_j)}{p(\sqrt{N} Y_{ij})} \prod_k \frac{p(\sqrt{N} Y_{kk}-\sqrt{\SNR N} x_k^2)}{p(\sqrt{N} Y_{kk})}\,.
\end{split} \eeq
As in the spiked Wigner case, the diagonal part can be handled separately, since it does not depend on the spike or the off-diagonal part.
Following the decomposition in \eqref{eq:decompose}, we set
\beq
	Q^{(s)}_{ij} := \frac{p^{(s)}(\sqrt{N} Y_{ij})}{p(\sqrt{N} Y_{ij})}.
\eeq
Then, we get
\beq \begin{split} \label{eq:iid_decompose}
	&\log \frac{1}{2^N} \sum_{\bsx} \prod_{i\neq j} \frac{p(\sqrt{N} Y_{ij}-\sqrt{\SNR N} x_i x_j)}{p(\sqrt{N} Y_{ij})} \\
	&= \log \frac{1}{2^N} \sum_{\bsx} \exp \left( \sum_{i < j} A^*_{ij} x_i x_j + \sum_{i < j} (B^*_{ij}+C^*_{ij}) + \caO(N^{-\frac{1}{2}}) \right),
\end{split} \eeq
where we let
\beq \begin{split}
	& A^*_{ij}:= -\sqrt{\SNR N} \bigg( Q^{(1)}_{ij} + Q^{(1)}_{ji} 
	+ \frac{\SNR}{6N} \Big( Q^{(3)}_{ij} + Q^{(3)}_{ji} - 3 Q^{(1)}_{ij} Q^{(2)}_{ij} - 3 Q^{(1)}_{ji} Q^{(2)}_{ji} + 2 (Q^{(1)}_{ij})^3 + 2 (Q^{(1)}_{ji})^3 \Big) \bigg), \\
	& B^*_{ij}:= \frac{\SNR}{2N} \left( Q^{(2)}_{ij} + Q^{(2)}_{ji} - (Q^{(1)}_{ij})^2 - (Q^{(1)}_{ji})^2 \right), \\
	& C^*_{ij}:= \frac{\SNR^2}{24N^2} \bigg( Q^{(4)}_{ij} + Q^{(4)}_{ji} -3 (P^{(2)}_{ij})^2 -3 (P^{(2)}_{ji})^2 - 4Q^{(1)}_{ij} Q^{(3)}_{ij} - 4Q^{(1)}_{ji} Q^{(3)}_{ji} \\
	&\qquad \qquad \qquad + 12(Q^{(1)}_{ij})^2 Q^{(2)}_{ij} + 12(Q^{(1)}_{ji})^2 Q^{(2)}_{ji} -6(P^{(1)}_{ij})^4 -6(P^{(1)}_{ji})^4 \bigg).
\end{split} \eeq

We first consider the spin glass part and prove a counterpart of Proposition \ref{prop:Z_limit}. Set
\beq \label{eq:Z_star}
	Z^*:= \frac{1}{2^N} \sum_{\bsx} \exp \left( \sum_{i<j} A^*_{ij} x_i x_j \right).
\eeq
Then, there exist random variables $\zeta^*$ and $(\zeta^*)'$ such that
\beq
	\log Z^* = \log \zeta^* + (\zeta^*)' + \caO(N^{-1}),
\eeq
where $\log \zeta^*$ and $(\zeta^*)'$ are asymptotically orthogonal to each other under $L_{B^*}^2$ and satisfy the following: The conditional distribution of $\log \zeta^*$ given $B^*$ converges in distribution to $\caN (-\nu^*, 2\nu^*)$, with 
\beq
	\nu^*:= \sum_{k=3}^\infty \frac{(2\SNR F)^k}{4k} = -\frac{1}{4} \left( \log (1- 2\SNR F) + 2\SNR F + 2 \SNR^2 F^2 \right).
\eeq
Further,
\beq
	(\zeta^*)' = \frac{1}{2N^2} \sum_{i<j} \E_{B^*} [(A^*_{ij})^2] - \frac{\SNR^2}{12} \E[ (Q^{(1)}_{12})^4 ] - \frac{\SNR^2}{4} \E[ (Q^{(1)}_{12})^2 ]^2 + U^*,
\eeq
where $U^*$ is a random variable whose asymptotic law is a centered Gaussian with variance 
\beq
	\theta^* := \frac{\SNR^2}{8} \E \left[ \Var_{B^*} \big( (Q^{(1)}_{12} + Q^{(1)}_{12})^2 \big) \right].
\eeq

We next follow the analysis in Section \ref{subsec:CLT} for the CLT part. The terms involving $B^*$ in the right hand side of \eqref{eq:iid_decompose} are
\[
	\sum_{i<j} B^*_{ij} + \frac{1}{2N^2} \sum_{i<j} \E[(A^*_{ij})^2|B^*_{ij}].
\]
By definition,
\[
	\E[B^*_{ij}] = \frac{\SNR}{N} \E \left[ Q^{(2)}_{ij} - (Q^{(1)}_{ij})^2 \right] = -\frac{\SNR F}{N}
\]
and
\beq \begin{split}
	&\E \left[ \E[(A^*_{ij})^2|B^*_{ij}] \right] = \E[(A^*_{ij})^2] \\
	&= 2\SNR N \E\left[ (Q^{(1)}_{ij})^2 \right] + \frac{2\SNR^2}{3} \E\left[ Q^{(1)}_{ij} Q^{(3)}_{ij} - 3 (Q^{(1)}_{ij})^2 Q^{(2)}_{ij} + 2 (Q^{(1)}_{ij})^4 \right] + O(N^{-1}) \\
	&= 2\SNR N \E\left[ (Q^{(1)}_{ij})^2 \right] + \frac{2\SNR^2}{3} \E\left[ Q^{(1)}_{ij} Q^{(3)}_{ij} \right] + O(N^{-1}),
\end{split} \eeq
where we used the independence of $Q^{(s)}_{ij}$ and $Q^{(s)}_{ji}$ and also that $\E[Q^{(1)}_{ij}] = 0$. We thus get
\beq \begin{split} \label{eq:B_star_mean}
	\E \left[ \sum_{i<j} B^*_{ij} + \frac{1}{2N^2} \sum_{i<j} \E[(A^*_{ij})^2|B^*_{ij}] \right] = \frac{\SNR^2}{6} \E\left[ Q^{(1)}_{12} Q^{(3)}_{12} \right] + O(N^{-1}).
\end{split} \eeq
The computation of the variance is similar, and we only remark important changes here as follows:
\[
	\E[(B^*_{ij})^2] = \frac{\SNR^2}{2N^2} \E \left[ (Q^{(2)}_{ij})^2 -2 (Q^{(1)}_{ij})^2 Q^{(2)}_{ij} + (Q^{(1)}_{ij})^4 \right] + \frac{\SNR^2 F^2}{2N^2},
\]
\[
	\frac{1}{N^2} \E[(A^*_{ij})^2 B^*_{ij}] = \frac{\SNR^2}{N^2} \E \left[ (Q^{(1)}_{ij})^2 \Big( Q^{(2)}_{ij} - (Q^{(1)}_{ij})^2 \Big) \right] - \frac{\SNR^2 F^2}{N^2} + O(N^{-3}),
\]
and
\[
	\frac{1}{4N^4} \E[ (\E[(A^*_{ij})^2|B^*_{ij}])^2] = \frac{\SNR^2}{4N^2} \E[ (\E[(Q^{(1)}_{ij} + (Q^{(1)}_{ji})^2|B^*_{ij}])^2] + O(N^{-3}).
\]
Using the identities
\[
	\E[(Q^{(1)}_{ij} + Q^{(1)}_{ji})^4] = \E \left[ (\E[(Q^{(1)}_{ij} + Q^{(1)}_{ji})^2|B^*_{ij}])^2 + \Var \big( (Q^{(1)}_{ij} + Q^{(1)}_{ji})^2|B^*_{ij} \big) \right]
\]
and
\[
	\E[(Q^{(1)}_{ij} + Q^{(1)}_{ji})^4] = 2 \E[(Q^{(1)}_{ij})^4] + 6 \E[(Q^{(1)}_{ij})^2]^2,
\]
we get
\beq \begin{split} \label{eq:B_star_var}
	\E\left[ \left( B^*_{ij} + \frac{\E[(A^*_{ij})^2|B^*_{ij}]}{2N^2} \right)^2 \right] 
	= \frac{\SNR^2}{4N^2} \E \left[ 2(Q^{(2)}_{ij})^2 - \Var \big( (Q^{(1)}_{ij})^2|B^*_{ij} \big) \right] + \frac{\SNR^2 F^2}{N^2} + O(N^{-3}).
\end{split} \eeq

For the LLN part, we simply get
\beq \label{eq:iid_LLN}
	\sum_{i<j} C^*_{ij} = -\frac{\SNR^2}{24} \E[Q^{(1)}_{12} Q^{(3)}_{12}] + \caO(N^{-1}).
\eeq
The computation for the diagonal part coincides with that for the spiked Wigner case, and we get
\beq \label{eq:iid_diag}
	\log \prod_k \frac{p(\sqrt{N} Y_{kk}-\sqrt{\SNR N} x_k^2)}{p(\sqrt{N} Y_{kk})} \Rightarrow \caN (-\frac{\SNR F}{2}, \SNR F).
\eeq

Combining Equations \eqref{eq:iid_LR}, \eqref{eq:Z_star}, \eqref{eq:B_star_mean}, \eqref{eq:B_star_var}, \eqref{eq:iid_LLN}, and \eqref{eq:iid_diag}, we find that $\log \caL(Y;\lambda)$ converges to the Gaussian whose mean is
\beq \begin{split}
	&\frac{1}{4} \left( \log (1-2\SNR F) + 2\SNR F+ 2\SNR^2 F^2 \right) - \frac{\SNR^2}{12} \E[ (Q^{(1)}_{12})^4 ] - \frac{\SNR^2 F^2}{4} 
	+ \frac{\SNR^2}{6} \E[ Q^{(1)}_{12} Q^{(3)}_{12}] -\frac{\SNR^2}{24} \E[Q^{(1)}_{12} Q^{(3)}_{12}] -\frac{\SNR F}{2} \\
	&= \frac{1}{4} \left( \log (1-\SNR F) + \SNR^2 F^2 \right) - \frac{\SNR^2}{8} \E[ (Q^{(2)}_{12})^2] = -\rho^*
\end{split} \eeq
and variance
\beq \begin{split}
	&-\frac{1}{2} \left( \log (1-2\SNR F) + 2\SNR F+ 2\SNR^2 F^2 \right) + \frac{\SNR^2}{8} \E\left[\Var_{B^*} \big( (Q^{(1)}_{12} + Q^{(1)}_{21})^2 \big) \right] \\
	&\qquad \qquad + \frac{\SNR^2}{8} \E \left[ 2(Q^{(2)}_{12})^2 - \Var_{B^*} \big( (Q^{(1)}_{12} + Q^{(1)}_{21})^2 \big) \right] + \frac{\SNR^2 F^2}{2} + \SNR F = 2\rho^*.
\end{split} \eeq
This proves Theorem \ref{thm:iid}.

\end{appendix}


\begin{thebibliography}{10}

\bibitem{ALR87}
M.~Aizenman, J.~L. Lebowitz, and D.~Ruelle.
\newblock Some rigorous results on the {S}herrington-{K}irkpatrick spin glass
  model.
\newblock {\em Comm. Math. Phys.}, 112(1):3--20, 1987.

\bibitem{BBP2005}
J.~Baik, G.~Ben~Arous, and S.~P\'ech\'e.
\newblock Phase transition of the largest eigenvalue for nonnull complex sample
  covariance matrices.
\newblock {\em Ann. Probab.}, 33(5):1643--1697, 2005.

\bibitem{Baik-Lee2016}
J.~Baik and J.~O. Lee.
\newblock Fluctuations of the free energy of the spherical
  {S}herrington-{K}irkpatrick model.
\newblock {\em J. Stat. Phys.}, 165(2):185--224, 2016.

\bibitem{NIPS_Barbier}
J.~Barbier, M.~Dia, N.~Macris, F.~Krzakala, T.~Lesieur, and L.~Zdeborov\'{a}.
\newblock Mutual information for symmetric rank-one matrix estimation: A proof
  of the replica formula.
\newblock In {\em Advances in Neural Information Processing Systems 29}, pages
  424--432. 2016.

\bibitem{Raj2011}
F.~Benaych-Georges and R.~R. Nadakuditi.
\newblock The eigenvalues and eigenvectors of finite, low rank perturbations of
  large random matrices.
\newblock {\em Adv. Math.}, 227(1):494--521, 2011.

\bibitem{benaych2012singular}
F.~Benaych-Georges and R.~R. Nadakuditi.
\newblock The singular values and vectors of low rank perturbations of large
  rectangular random matrices.
\newblock {\em J. Multivariate Anal.}, 111:120--135, 2012.

\bibitem{CapitaineDonatiFeral2009}
M.~Capitaine, C.~Donati-Martin, and D.~F\'eral.
\newblock The largest eigenvalues of finite rank deformation of large {W}igner
  matrices: convergence and nonuniversality of the fluctuations.
\newblock {\em Ann. Probab.}, 37(1):1--47, 2009.

\bibitem{chen2021asymmetry}
Y.~Chen, C.~Cheng, and J.~Fan.
\newblock Asymmetry helps: Eigenvalue and eigenvector analyses of
  asymmetrically perturbed low-rank matrices.
\newblock {\em Ann. Stastist}, 49(1):435, 2021.

\bibitem{chung2019weak}
H.~W. Chung and J.~O. Lee.
\newblock Weak detection of signal in the spiked wigner model.
\newblock In {\em Proceedings of the 36th International Conference on Machine
  Learning}, volume~97, pages 1233--1241, 2019.

\bibitem{Dobriban2017}
E.~Dobriban.
\newblock Sharp detection in {PCA} under correlations: {A}ll eigenvalues
  matter.
\newblock {\em Ann. Statist.}, 45(4):1810--1833, 2017.

\bibitem{edwards1975theory}
S.~F. Edwards and P.~W. Anderson.
\newblock Theory of spin glasses.
\newblock {\em J. Phys. F}, 5(5):965, 1975.

\bibitem{el2018detection}
A.~El~Alaoui and M.~I. Jordan.
\newblock Detection limits in the high-dimensional spiked rectangular model.
\newblock In {\em Conference On Learning Theory}, pages 410--438, 2018.

\bibitem{AlaouiJordan2018}
A.~El~Alaoui, F.~Krzakala, and M.~I. Jordan.
\newblock {Fundamental limits of detection in the spiked {W}igner model}.
\newblock {\em Ann. Statist.}, 48(2):863--885, 2020.

\bibitem{FeralPeche2007}
D.~F\'eral and S.~P\'ech\'e.
\newblock The largest eigenvalue of rank one deformation of large {W}igner
  matrices.
\newblock {\em Comm. Math. Phys.}, 272(1):185--228, 2007.

\bibitem{Johnstone2001}
I.~M. Johnstone.
\newblock On the distribution of the largest eigenvalue in principal components
  analysis.
\newblock {\em Ann. Statist.}, 29(2):295--327, 2001.

\bibitem{johnstone2020testing}
I.~M. Johnstone and A.~Onatski.
\newblock Testing in high-dimensional spiked models.
\newblock {\em Ann. Stastist}, 48(3):1231--1254, 2020.

\bibitem{pmlr-v139-jung21a}
J.~H. Jung, H.~W. Chung, and J.~O. Lee.
\newblock Detection of signal in the spiked rectangular models.
\newblock In {\em Proceedings of the 38th International Conference on Machine
  Learning}, volume 139, pages 5158--5167, 2021.

\bibitem{Montanari2017}
A.~Montanari, D.~Reichman, and O.~Zeitouni.
\newblock On the limitation of spectral methods: from the {G}aussian hidden
  clique problem to rank one perturbations of {G}aussian tensors.
\newblock {\em IEEE Trans. Inform. Theory}, 63(3):1572--1579, 2017.

\bibitem{Onatski2013}
A.~Onatski, M.~J. Moreira, and M.~Hallin.
\newblock Asymptotic power of sphericity tests for high-dimensional data.
\newblock {\em Ann. Statist.}, 41(3):1204--1231, 2013.

\bibitem{Onatski2014}
A.~Onatski, M.~J. Moreira, and M.~Hallin.
\newblock Signal detection in high dimension: the multispiked case.
\newblock {\em Ann. Statist.}, 42(1):225--254, 2014.

\bibitem{parisi1980sequence}
G.~Parisi.
\newblock A sequence of approximated solutions to the sk model for spin
  glasses.
\newblock {\em J. Phys. A}, 13(4):L115, 1980.

\bibitem{Peche2006}
S.~P\'ech\'e.
\newblock The largest eigenvalue of small rank perturbations of {H}ermitian
  random matrices.
\newblock {\em Probab. Theory Related Fields}, 134(1):127--173, 2006.

\bibitem{Perry2018}
A.~Perry, A.~S. Wein, A.~S. Bandeira, and A.~Moitra.
\newblock Optimality and sub-optimality of {PCA} {I}: {S}piked random matrix
  models.
\newblock {\em Ann. Statist.}, 46(5):2416--2451, 2018.

\bibitem{sherrington1975solvable}
D.~Sherrington and S.~Kirkpatrick.
\newblock Solvable model of a spin-glass.
\newblock {\em Phys. Rev. Lett.}, 35(26):1792, 1975.

\bibitem{talagrand2006parisi}
M.~Talagrand.
\newblock The {P}arisi formula.
\newblock {\em Ann. of Math. (2)}, 163(1):221--263, 2006.

\end{thebibliography}
\end{document}